\documentclass[a4paper,11pt]{article}
\pdfoutput = 1
\usepackage[latin1]{inputenc}
\usepackage{t1enc,a4wide}
\usepackage[english]{babel}
\usepackage{graphicx,color}
\usepackage{amsmath,amssymb,amsthm}
\usepackage[notref,notcite,final]{showkeys}
\usepackage[bookmarks=false]{hyperref}
\usepackage[numbers]{natbib}
\usepackage{enumitem}
\usepackage[linesnumbered,lined,boxed,algosection]{algorithm2e}
\DontPrintSemicolon

\graphicspath{{figs/}{./}}
\def\RR{{\mathbb R}}
\def\EE{{\mathbb E}}
\def\LL{{\mathbb L}}
\def\NN{{\mathbb N}}
\def\PP{{\mathbb P}}

\def\t{\theta}
\def\s{\star}
\def\la{\lambda}
\def\l{\ell}
\def\tX{{\tilde X}}
\def\dt{\, dt}

\newcommand{\Ec}{{\cal E}}
\newcommand{\Fc}{{\cal F}}
\newcommand{\Gc}{{\cal G}}
\newcommand{\Hc}{{\cal H}}

\newcommand{\Mc}{{\cal M}}
\newcommand{\Nc}{{\cal N}}

\newcommand{\Vc}{{\cal V}}
\newcommand{\Var}{\mathrm{Var}}
\newcommand{\Cov}{\mathrm{Var}}

\newcommand{\Tr}{\mathrm{Tr}}
\newtheorem{thm}{Theorem}[section]
\newtheorem{theorem}[thm]{Theorem}

\newtheorem{proposition}[thm]{Proposition}
\newtheorem{lemma}[thm]{Lemma}
\newtheorem{corollary}[thm]{Corollary}

\newtheorem{remark}[thm]{Remark}
\def\expp#1{\mathrm{e}^{#1}}
\def\inv#1{\mathop{\frac{1}{ #1}}\nolimits}
\DeclareMathOperator*{\argmin}{\arg\!\min}
\def\ind#1{\mathrm{1}_{\left\{#1\right\}}}

\def\abs#1{\left|#1\right|}
\def\defeq{\stackrel{\Delta}{=}}

\numberwithin{equation}{section}
\definecolor{forest}{rgb}{0.13,0.55,0.13}

\makeatletter
\newcounter{hypo}
\renewcommand{\thehypo}{(${\mathcal H}$-\arabic{hypo})}
\newcommand*{\dohypo}{\thehypo}
\newenvironment{hypo}[1][]{%
\refstepcounter{hypo}
\list{}{%
\settowidth{\labelwidth}{\dohypo}%
\setlength{\labelsep}{10pt}%
\setlength{\leftmargin}{\labelwidth}
\advance\leftmargin\labelsep%
}%
\item[\dohypo  #1]%
  }{%
\endlist
}

\makeatother

\begin{document}

\title{Coupling Importance Sampling and Multilevel Monte Carlo using Sample Average
  Approximation}
\author{Ahmed Kebaier\footnote{Universit\'e Paris 13, Sorbonne Paris Cit\'e, LAGA, CNRS (UMR
    7539), kebaier@math.univ-paris13.fr. This research benefited from the support of the
  chair Risques Financiers, Fondation du Risque and the Laboratory of Excellence MME-DII
(http://labex-mme-dii.u-cergy.fr/).}\; \& J\'er\^ome Lelong\footnote{Univ. Grenoble Alpes,
  Laboratoire Jean Kuntzmann, jerome.lelong@univ-grenoble-alpes.fr.  This project was supported by the Finance for Energy Market Research Centre, www.fime-lab.org.
}}

\maketitle

\begin{abstract}
  In this work, we propose a smart idea to couple importance sampling and Multilevel Monte
  Carlo (MLMC). We advocate a per level approach with as many importance sampling parameters as
  the number of levels, which enables us to handle the different levels independently.
  The search for parameters is carried out using sample average approximation, which
  basically consists in applying deterministic optimisation techniques to a Monte Carlo
  approximation rather than resorting to stochastic approximation.  Our innovative
  estimator leads to a robust and efficient procedure reducing both the discretization
  error (the bias) and the variance for a given computational effort.  In the setting of
  discretized diffusions, we prove that our  estimator satisfies a strong law of large
  numbers and a central limit theorem with optimal limiting variance, in the sense that
  this is the variance achieved by the best importance sampling measure (among the class
  of changes we consider), which is however non tractable. Finally, we illustrate the
  efficiency of our method on several numerical challenges coming from quantitative
  finance and show that it outperforms the standard MLMC estimator.

\vspace{0.2cm}


\end{abstract}

\section{Introduction}

Expectations involving a stochastic process are often computed using a Monte Carlo method
combined with a discretization scheme. For instance, computing a hedging portfolio in
finance uses these tools.  Generally,  the asset price is modeled by a diffusion process
$(X_t)_{0\leq t\leq T}$, defined as the solution of a stochastic differential equation
(SDE)
\begin{equation}\label{eq:X}
  dX_{t} = b(X_{t})dt + \sigma(X_{t}) dW_{t},\quad X_0=x\in \mathbb R^d
\end{equation}
where $b : \RR^d \to \RR^d$, $\sigma : \RR^d \to \Mc_{d \times q}$ and $W$ is a Brownian
motion with values in $\RR^q$ defined on some given probability space $(\Omega,(\mathcal
F_t)_{0 \leq t \le T},\PP)$ with finite time horizon $T>0$. The process $X$ hardly ever
has an explicit solution, which implies that its simulation requires the use of a
discretization scheme. For $n \in \NN^*$, consider the continuous time Euler approximation
$X^n$ with time step $\delta=T/n$ given by
\begin{equation*}
  dX^n_t=b(X^n_{\eta_n(t)})dt+
  \sigma(X^n_{\eta_n(t)})dW_t, \quad \eta_n(t)=\lfloor t/\delta \rfloor \delta.
\end{equation*}
This work aims at combining importance sampling with different discretization methods:
first, we study the use of importance sampling for the standard case of Euler Monte Carlo and
then we apply it to MLMC. Many different changes of measure can be used
to implement importance sampling. When working with Lévy processes, it is common to use the
Esscher transform to introduce a new family of measures. For Brownian driven SDEs, the
Esscher transform actually corresponds to a Gaussian change of measure in the spirit of
the Girsanov theorem. Following the ideas of Arouna \cite{Aro}, we consider a parametric
family of stochastic processes $(X_t({\t}))_{0\leq t\leq T}$, with $\t \in \RR^q$, driven
by a Brownian motion with linear drift
\begin{equation*}
  dX_{t}(\t)=\left(b(X_{t}(\t)) + \sigma(X_{t}(\t)) \t \right) dt+\sigma(X_{t}(\t)) dW_{t}.
\end{equation*}
We also define the continuous time Euler approximation $X^n(\t)$ of the process $X(\t)$.
From Girsanov's Theorem, the process $ (B^{\t}_{t} \defeq W_{t} + \t t)_{t \le T}$ is a
Brownian motion under the new probability measure $\PP_\t$ equivalent to $\PP$ and such
that
\begin{equation*}
  \frac{d\PP_{\t}}{d\PP}_{|\Fc_{t}}=\exp\left(-\t\cdot W_{t} -
  \frac{1}{2}|\t|^{2}t\right) \defeq \Ec^-(W, \t).
\end{equation*}
Therefore,
\begin{align}
  \label{eq:is}
  \EE_{\PP}[\psi(X_T)] = \EE_{\PP_\t}[\psi(X_T(\t))] =
  \EE_{\PP}\left[\psi(X_T(\t)) \Ec^-(W,\t)\right].
\end{align}
This equality still holds when replacing $X$ (resp. $X(\t)$) by its Euler scheme $X^n$
(resp. $X^n(\t)$).  The l.h.s. and r.h.s. expectations are both computed under the
original probability measure. In the following, we will always use the measure $\PP$ and
therefore we will not write it anymore.  The idea of importance sampling Monte Carlo is to
use the r.h.s of~\eqref{eq:is} to build a Monte Carlo estimator of $\EE[\psi(X_T)]$ using
$X^n(\t)$ with $\t$ given by
$$ 
\t^\s = \argmin_{\t \in \RR^q} \Var\left(\psi(X_T(\t)) \Ec^-(W, \t)\right).
$$
Importance sampling for Euler Monte Carlo is studied in Section~\ref{sec:euler}: first, we
investigate how to approximate $\t^\s$ in practice and second we prove that the Monte
Carlo estimator combined with this approximation of $\t^\s$ satisfies a strong law of
large numbers and a central limit theorem when both $n$ and the number of samples go to
infinity. This result extends the limit theorems obtained in~\cite{JourLel}, in which the
authors investigated the case of a fixed number of discretization steps $n$.  The error
induced by using $\EE[\psi(X^n_T(\t))]$ instead of $\EE[\psi(X_T(\t))]$  is called the
discretization error and is responsible for the bias of the Euler Monte Carlo estimator,
while the Monte Carlo approximation only impacts the variance. The two errors are balanced
when the number of samples $N$ of the Monte Carlo method is proportional to $n^2$, which
leads to an overall complexity of order $n^3$.  In order to reduce the bias for a given
computational effort, Kebaier \cite{Keb} proposed to use the Statistical Romberg method,
which combines discretization schemes on two nested time grids. This method was
generalized by Giles \cite{Gil} who proposed to use a multilevel Monte Carlo algorithm
following the line of Heinrich's multilevel method for parametric integration~\cite{hei}. 

Let $m, L \in \NN$  with $m \ge 2$ and $L > 0$, the idea of the multilevel method is to
write the expectation on the finest time grid as a telescopic sum involving all the other
grids (referred to as levels)
\begin{align}
  \label{eq:ml-tele}
  \EE[\psi(X_T^{m^L})] = \EE[\psi(X_T^{m^0})] + \sum_{\l=1}^L \EE[\psi(X_T^{m^\l}) -
  \psi(X_T^{m^{\l - 1}})]
\end{align}
and then to approximate each expectation by a Monte Carlo method with a well chosen number
of samples to balance the errors between the different terms.  We refer the reader to the
extensive literature on MLMC for more details, see e.g.  \cite{BAK2014,tempone,
Cre,Der, GilMil, GilHig, GilSzp,Heibis, HeiSin, LemPag14}.  For a fixed computational
budget, the use of multilevel techniques clearly reduces the bias, but in many
situations the high variance also brings in a significant inaccuracy, which naturally
leads to trying to couple MLMC with variance reduction techniques.

In this work, we focus on coupling importance sampling with MLMC.  In \cite{BHK13-1} and
\cite{Haj}, the authors choose to apply MLMC to the right hand side of~\eqref{eq:is} coming
up with
\begin{equation}
  \label{eq:is-ml}
  \begin{split}
    \EE[\psi(X_T^{m^L})] & = \EE\left[\psi(X_T^{m^0}(\la)) \Ec^-(W, \la)\right] \\
    & \quad + \sum_{\l=1}^L \EE\left[(\psi(X_T^{m^\l}(\la)) -
    \psi(X_T^{m^{\l - 1}}(\la))) \Ec^-(W, \la)\right].
  \end{split}
\end{equation}
This approach mixes all the levels through the optimization of the parameter $\la$ and
breaks the independence between the levels of the multilevel approach, which nonetheless
made it so popular and easy to implement.

Instead of using~\eqref{eq:is-ml}, we would rather apply importance sampling to each
expectation in the telescopic sum of~\eqref{eq:ml-tele} to obtain for $\la_1, \dots, \la_L
\in \RR^q$
\begin{align*}
  \EE[\psi(X_T^{m^L})] & = \EE\left[\psi(X_T^{m^0}(\la_0)) \Ec^-(W, \la_0)\right] \\
  &\quad + \sum_{\l=1}^L \EE\left[(\psi(X_T^{m^\l}(\la_\l)) -
\psi(X_T^{m^{\l - 1}}(\la_\l))) \Ec^-(W, \la_\l)\right].
\end{align*}
Our importance sampling multilevel estimator is obtained by applying a Monte Carlo method to each
of the levels $\l$ with $N_\l$ samples
\begin{align}
  \label{eq:estimator-intro}
  Q_L(\la_0,\dots,\la_L) & =\frac{1}{N_0}\sum_{k=1}^{N_0}
  \psi( \tX^{m^0}_{T,0,k}(\la_0)) \Ec^-(\tilde W_{0,k}, \la_0) \nonumber \\
  & \qquad +\sum_{\ell=1}^{L}\frac{1}{N_\ell}
  \sum_{k=1}^{N_\ell}\left(\psi( \tX^{m^{\ell}}_{T,\ell,k}(\la_{\ell}))-\psi(
    \tX^{m^{\ell-1}}_{T,\ell,k}(\la_{\ell}))\right) \Ec^-(\Tilde W_{\ell,k}, \la_\ell)
\end{align}
The samples used in the different levels are independent and within each level they are
i.i.d. For any $\l \ge 0$, the variables $\tX^{m^{\ell}}_{T,\l,k}(\lambda_\l)$ (resp.
$\tX^{m^{\ell-1}}_{T,\l,k}(\lambda_\l)$ when $\l > 0$) are the terminal values of the
Euler schemes of $X(\lambda_\l)$ with $m^\l$ (resp. $m^{\l-1}$) time steps built using
the same Brownian path $\tilde W_{\l, k}$.  The variance of the importance sampling MLMC
estimator is given by 
$$
\Var[Q_L]= N_0^{-1} \sigma_0(\la_0)^2 +\sum_{\ell=1}^L N^{-1}_{\ell} \frac{(m-1)T}{m^\l}
\sigma_{\ell}^2(\lambda_{\ell})
$$
where 
\begin{align*}
  \sigma_0^2 (\la_0)  & \defeq \Var[\psi(X^{m^0}_{T}(\lambda_0)) \Ec^-(W,\lambda_0)] \\
  \sigma_{\ell}^2(\lambda_\l) & \defeq \frac{m^\l}{(m-1) T} \Var\left[\left(\psi(
      X^{m^{\ell}}_{T}(\lambda_\l)) - \psi(
  X^{m^{\ell-1}}_{T}(\lambda_\l))\right)\Ec^-(W, \la_\l) \right].
\end{align*}
By allowing for one importance sampling parameter $\la_\l$ per level, our approach has
many advantages over~\cite{BHK13-1, Haj}.  First, the computations within the different
levels remain independent. Second, the variance of each level $\l$ only depends on
$\lambda_\l$, which reduces the global minimization problem to several smaller
minimization problems.  Third, we actually minimize the real variance of the estimator and
not its asymptotic value and more importantly it can be implemented without knowing
$\nabla \psi$, which however appears in the central limit theorem for MLMC. The new idea
of using one importance sampling parameter per level was later taken up in~\cite{BHK16}
but coupled with stochastic approximation to build adaptive estimators.

Actually, minimizing $\la \longmapsto \sigma_{\ell}^2(\lambda)$ can be achieved by using
the randomly truncated Robbins Monro algorithm  proposed by Chen et al. \cite{CHZY,CGG}
and later investigated in the context of importance sampling by Lapeyre and Lelong
\cite{LapLel} and Lelong \cite{Lel2008}. The numerical stability of these stochastic
algorithms strongly depends on the choice of the descent step --- often referred to as the
gain sequence --- which proves to be highly sensitive in practice.  To overcome this
difficulty, Jourdain and Lelong \cite{JourLel} proposed to apply deterministic
optimization techniques to sample average estimators to search for the optimal parameter.
Following their methodology, we define $\sigma_{\l, N_\l'}^2$ as the sample average
approximation of $\sigma_\l^2$ with $N_\l'$ samples using the standard empirical Monte
Carlo estimator of the variance. We assume that the samples used in $\sigma_{\l, N_\l'}^2$
are independent of those used in $Q_L$. We refer to Section \ref{sec:mul} for more details
on how to choose the samples in the different approximations.  Now, we sketch the algorithm
corresponding to our method.
\begin{algorithm}[ht]
  \For{$\l=0:L$}{
    Sample the random function $\lambda \longmapsto \sigma_{\l, N_\l'}(\lambda)$.
    \textcolor{forest}{\tcp*[r]{$\sigma_{\l, N_\l'}^2$ is the sample average approximation of $\sigma_{\l}^2$, see
    Section~\ref{sec:ml-framework}}}
    Compute  $\hat \lambda_\l = \argmin \sigma_{\l, N_\l'}^2(\lambda)$ using
    Newton--Raphson's algorithm. \;
    Independently of $\sigma_{\l, N_\l'}^2$, sample the level $\l$
    of~\eqref{eq:estimator-intro} using $\hat \la_\l$. \;
  }
  Sum all the levels to obtain
  \begin{align*}
      Q_L(\hat \la_0,\dots,\hat \la_L) & =\frac{1}{N_0}\sum_{k=1}^{N_0}
      \psi( \tX^{m^0}_{T,0,k}(\hat \la_0)) \Ec^-(\tilde W_{0,k}, \hat \la_0) \\
      & \qquad +\sum_{\ell=1}^{L}\frac{1}{N_\ell}
      \sum_{k=1}^{N_\ell}\left(\psi( \tX^{m^{\ell}}_{T,\ell,k}(\hat \la_{\ell}))-\psi(
        \tX^{m^{\ell-1}}_{T,\ell,k}(\hat \la_{\ell}))\right) \Ec^-(\Tilde W_{\ell,k},
      \hat \la_\ell).
    \end{align*}
  \caption{Multilevel Importance Sampling (MLIS)}
  \label{algo:MLIS-sketch}
\end{algorithm}

First, we investigate in Section~\ref{sec:euler} the standard Euler Monte Carlo method
coupled with importance sampling. Then, in Section~\ref{sec:mul}, we study the importance
sampling framework with MLMC.  We prove that $Q_L(\hat \la_0,\dots,\hat \la_L)$ satisfies
a strong law of large numbers and a central limit theorem. Our MLIS estimator achieves the
smallest possible variance within the family of MLMC estimators approximating
$\EE[\psi(X_T)]$ using the class of processes $(X(\la))_{\la \in \RR^q}$.  Note that this
is also the limiting variance obtained in~\cite{BHK13-1} for the MLMC estimator built
on~\eqref{eq:is-ml} with the best possible parameter $\lambda \in \RR^q$.  The main
difficulty in proving these results is the uniform control of the triangular arrays
involved in the adaptive multilevel estimator. To overcome this issue, we prove in
Section~\ref{sec:slln} new limit theorems for doubly indexed sequences of random variables
in a general setting (see Propositions \ref{prop:slln2} and \ref{prop:slln2-u}). In
section \ref{sec:num}, we illustrate the efficiency of MLIS on challenging problems coming
from quantitative finance and show that it outperforms the standard MLMC estimator.

\section{Importance sampling with Euler Monte Carlo}
\label{sec:euler}
\subsection{Notation and general assumptions}

\begin{itemize}
  \item For a vector $x \in \RR^q$, $\abs{x}$ denotes the Euclidean norm of $x$.
  \item The superscript ${}^*$ denotes the transpose operator.
  \item For a matrix $A \in \Mc_{d \times q}$, $\abs{M}$ denotes the Frobenius norm of
    $A$ defined by $\sqrt{\Tr(A^* A)}$, which corresponds to the Euclidean norm on
    $\RR^{d \times q}$.
  \item For $q \in \NN^*$, $I_q$ denotes the identity matrix with size $q \times q$. 
  \item For $\alpha >0$, we define the set of functions
    \begin{align}
      \label{eq:Halpha}
      \Hc_\alpha = \Big\{ & \psi : \RR^d \rightarrow \RR \mbox{ s.t. }  \exists c
        >0, \beta \ge 1 , \;  \forall x \in \RR^d , \; |\psi(x)| \le c
        (1 + |x|^\beta) \nonumber\\
        & \quad \mbox{and } \forall x,y  \in \RR^d , \; |\psi(x) -
      \psi(y)| \le c (1 + (|x|^\beta \wedge |y|^\beta)) |x - y|^\alpha \Big\}
    \end{align}
  \item For a sequence of random variables $(X_n)_n$, ``$X_n \Longrightarrow X$'' means
    that $(X_n)_n$ converges in distribution to $X$. 
\end{itemize}

Here, we gather several standard assumptions required to ensure the convergence of the
Euler scheme.

\begin{hypo}
  \label{A:general}
  \begin{subhypo}
  \item The functions $b$ and $\sigma$ are Lipschitz
      \begin{equation*}
        \label{eq:Hbsigma}   \quad \forall x,y\in\mathbb R^d, \quad
        |b(x)-b(y)|+ \abs{\sigma(x)-\sigma(y)}\leq
        C_{b,\sigma}|x-y|, \tag{$\mathcal H_{b,\sigma}$}
      \end{equation*}
      for some real number $C_{b,\sigma}>0$.
    \item \label{eq:P} $\forall p\geq 1,\;X,X^n\in L^p$ and there exists $K_p(T) >0$ s.t.
      \begin{equation*}
       \EE\left[{\sup_{0 \leq t \leq T}}\abs{X_t -  {X}^n_t}^{p} \right]
      \leq\frac{K_p(T)}{n^{p/2}}. 
    \end{equation*}
  \item There exist $\gamma \in [1/2, 1]$ and $C_{\psi}(T,\gamma) >0$ s.t.
    \begin{equation*}
      \label{eq:Hen}
       n^{\gamma}(\EE\psi(X^n_T)-\EE\psi(X_T))\rightarrow 
       C_\psi(T,\gamma). \tag{$\Hc_\gamma$}
    \end{equation*}
  \end{subhypo}
\end{hypo}

\begin{hypo}
  \label{A:nonul}
  The function $\psi$ satisfies
  \begin{equation}
    \label{eq:nonzero-int}
    \PP(\psi(X_T) \ne 0) > 0 \quad \text{and} \quad \forall \, \t \in \RR^q, \; \EE\left[
    \psi(X_T)^2 \expp{-\t \cdot W_T} \right] < \infty.
  \end{equation}
\end{hypo}

\subsection{General framework}

In this section, we investigate the case of an Euler Monte Carlo approach. We consider the importance
sampling representation of $\EE[\psi(X_T)]$ given by
\begin{align*}
  \EE[\psi(X_T(\t)) \Ec^-(W, \t)].
\end{align*}
The optimal value for $\t$ is given by
\begin{align*}
  \t^\s = \argmin_{\t \in \RR^q} \quad \mbox{ with } \quad v(\t) \defeq \EE[(\psi(X_T(\t))
  \Ec^-(W, \t))^2].
\end{align*}
By using~\eqref{eq:is}, we can rewrite $v$ as
\begin{align*}
  v(\t) = \EE[\psi(X_T)^2 \Ec^+(W, \t)]  \quad \mbox{ with } \quad \Ec^+(W, \t) \defeq
  \expp{-W_T \cdot \t + \frac{\abs{\t}}{2}}.
\end{align*}
From a practical point of view, the quantity $v(\t)$ is not explicit so we use the Euler
scheme to discretize  $X(\t)$ and approximate $\t^\s$ by
\begin{equation}
\label{eq_variance}
\t_n\defeq \argmin_{\t \in \RR^q} v_n(\t)\quad \mbox{ with }\quad   
v_n(\t)\defeq \EE\left[\psi(X_{T}^{n})^2 \Ec(W, \t)\right].
\end{equation}
Since the expectation is usually not tractable,
we replace it by its sample average approximation and define
\begin{equation}
  \label{eq_variance_saa}
  \t_{n,N}\defeq \argmin_{\t \in \RR^q} v_{n,N}(\t)\quad \mbox{ with }\quad  
  v_{n,N}(\t)\defeq \frac{1}{N}\sum_{i=1}^N\left(\psi(X^n_{T,i})^2 
  \Ec(W_{i}, \t)\right),
\end{equation}
where $(X^n_{T,i},W_{T,i})_{1\leq i\leq N}$ are i.i.d. samples with the law of
$(X^n_{T},W_{T})$. The existence and uniqueness of $\t^\s$, $\t_{n}$ and $\t_{n,N}$ are
ensured by the following lemma whose proof can easily be adapted from \cite[Lemma
1.1]{JourLel}.
\begin{lemma}
  \label{lem:convex} Under Condition~\ref{A:nonul}, the functions  $v$, $v_n$ and
  $v_{n, N}$ are infinitely continuously differentiable for all $n, N$ and the
  derivatives are obtained by exchanging expectation and differentiation.  Moreover, the
  functions $v$ and $v_n$ are strongly convex and so is $v_{n, N}$ for any $N$ such that
  $v_{n, N}$ is not identically zero.
\end{lemma}

\subsection{Convergence of the optimal importance sampling parameter}\label{sec:param}

\begin{theorem}
  \label{th:convergence} Suppose $\sigma$ and $b$ satisfy~\eqref{eq:Hbsigma}.  Let
  $\psi$ satisfy Condition~\ref{A:nonul} and belong to $\Hc_\alpha$ for
  some $\alpha>0$. Then, $\t_{n} \to \t^\s$ a.s. when $n \to +\infty$.
\end{theorem}
By Hölder's inequality, for any function $\psi \in \Hc_\alpha$, 
\ref{A:nonul} implies that $\sup_n\EE[\psi(X^n_T)^2 \expp{-\t \cdot
W_T}] < +\infty$.  Hence,  the proof of the theorem ensues from \cite[Theorem
2.2]{BHK13-1}. \\

In the following, we let $N$ depend on $n$ so that $N \defeq N_n$ tends to infinity
with $n$. \\

\begin{proposition}\label{prop:cvu}
  Assume that Assumption~\eqref{eq:Hbsigma} holds and that $\psi \in \Hc_\alpha$ for
  some $\alpha>0$. Then, for all $K >0$, a.s. when $n \to \infty$
  \begin{align*}
    \sup_{|\t| \le K}|v_{n,N_n}(\t)- v(\t)|\to 0; \;
    \sup_{|\t| \le K}|\nabla v_{n,N_n}(\t)- \nabla v
    (\t)|\to 0.
  \end{align*}
\end{proposition}

\begin{proof}
 The proof of the two results are very similar, we omit the second one and
 concentrate on the uniform convergence for $v_{n, N_n}$. To do so, we will
 apply Proposition~\ref{prop:slln2-u}. Now, we check Assumptions
 \ref{slln2-fm}, \ref{slln2-u}, \ref{slln2-sup-u}. At first, note that
 under Assumption \eqref{eq:Hbsigma}, we have the almost sure convergence of
 $X^n_{T}$ towards $X_{T}$. As $\psi\in \mathcal H_{\alpha}$, it follows from
 Property \ref{eq:P} that for all $a>1$,
 $\sup_{n\in\NN}\EE\left[\abs{\psi(X^n_{T})^2 e^{-\t \cdot
       W_{T}+\frac{1}{2}|\t|^{2}T}}^a\right]<\infty$. Note that for every fixed $n$,
 the sequence $\left(  \psi(X^n_{T,i})^2 e^{-\t \cdot
     W_{T,i}+\frac{1}{2}|\t|^{2}T}\right)_i$ is i.i.d.  Then, we deduce that
 for all $m\in\NN^*$
$$
\lim_{n\rightarrow\infty}\EE\left[ v_{n, m}(\t) \right]=\EE\left[\psi(X_{T})^2 
    e^{-\t \cdot W_{T}+\frac{1}{2}|\t|^{2}T}\right].
$$
This yields  \ref{slln2-fm}.  Let $K>0$.  As  $\psi\in \mathcal
H_{\alpha}$ we obtain using the Cauchy Schwarz inequality and Property  \ref{eq:P}  that
$$ 
\sup_n \sup_m m \Var\left(\sup_{|\t|\leq K} v_{n, m}(\t)\right)
\leq \sup_n\EE^{1/2}\left[\psi(X^n_{T})^8\right]
    \EE^{1/2}\left[ \sup_{|\t|\leq K}e^{-4\t \cdot W_{T}+2|\t|^{2}T}\right]<\infty. 
$$
Using the same arguments, we also get 
$$
\sup_n \sup_m \Var\left(\psi(X^n_{T,m})^2  \sup_{|\t|\leq K}  
    e^{-\t \cdot W_{T,m}+\frac{1}{2}|\t|^{2}T}   \right)<\infty.
$$
This yields \ref{slln2-u}.
Concerning the last assumption, if we fix $\delta>0$, $\t\in \RR^d$  and set
$B(\t, \delta)$ --- the open ball  with center $\t$ and radius $\delta$ --- then the
Cauchy Schwarz inequality gives
\begin{multline*}
  \sup_n\EE\left[ \psi(X^n_{T})^2\sup_{\t'\in B(\t, \delta)}\abs{e^{-\t' \cdot W_{T}+\frac{1}{2}|\t'|^{2}T}
- e^{-\t \cdot W_{T}+\frac{1}{2}|\t|^{2}T}} \right]^2\leq \\
\sup_n \EE\left[\psi(X^n_{T})^4\right] \EE\left[\sup_{\t'\in B(\t, \delta)}\abs{e^{-\t' \cdot W_{T}+\frac{1}{2}|\t'|^{2}T}
- e^{-\t \cdot W_{T}+\frac{1}{2}|\t|^{2}T}}^2 \right].
\end{multline*}
Using the elementary inequality $|e^x-e^y| \le \abs{x-y}
(\expp{{x}} + \expp{{y}})$, we easily deduce that the quantity $ \EE\left[\sup_{\t'\in
    B(\t, \delta)}\abs{e^{-\t' \cdot W_{T}+\frac{1}{2}|\t'|^{2}T} -
    e^{-\t \cdot W_{T}+\frac{1}{2}|\t|^{2}T}}^2 \right]$ can be made
arbitrarily small.  Finally, Assumption~\ref{slln2-sup-u} is satisfied using
Remark~\ref{rem:iid}.
\end{proof}

\begin{theorem}
  \label{thm:theta_cv} Assume that Assumption~\eqref{eq:Hbsigma} holds and that
  $\psi \in \Hc_\alpha$ for some $\alpha>0$. Then, $\t_{n,N_n}
  \longrightarrow \t^\s$ a.s. and 
  $\sqrt{N_n} (\t_{n,N_n} - \t^\s) \Longrightarrow N(0, \Gamma)$ when $n \to \infty$ with
  \begin{align*}
    \Gamma = [\nabla^2 v(\t^\s)]^{-1} \Cov\left( (T\t^\s - W_T) \psi(X_T)^2
    \expp{-\t^\s \cdot W_T + \frac{1}{2} |\t^\s|^2 T} \right) [\nabla^2
    v(\t^\s)]^{-1}.
  \end{align*}
\end{theorem}

\begin{proof}
  We already know from Proposition~\ref{prop:cvu} that a.s. $v_{n,N_n}$ converges
  locally uniformly to $v$. Let $\varepsilon >0$. By the strict convexity of
  $v$, $\delta \defeq \inf_{|\t - \t^\s| \ge \varepsilon} v(\t) - v(\t^\s)
  > 0$.  \\
  The local uniform convergence of $v_{n,N_n}$ to $v$ ensures that 
  \begin{align}
    \label{eq:luc}
    \exists n_\delta > 0, \forall n \ge n_\delta, \forall \t \in \RR^q\mbox{ s.t. } |\t -
    \t^\s| \le \varepsilon, \; |v_{n, N_n}(\t) - v(\t)| \le \frac{\delta}{3}.
  \end{align}
  For $n \ge n_\delta$ and $\t$ such that $|\t - \t^\s| \ge \varepsilon$, we can
  deduce from the convexity of $v_{n,N_n}$ that
  \begin{align*}
    v_{n,N_n}(\t) - v_{n,N_n}(\t^\s) & \ge \frac{|\t - \t^\s|}{\varepsilon} \left[
    v_{n,N_n}\left( \t^\s + \varepsilon \frac{\t - \t^\s}{|\t - \t^\s|} \right) -
    v_{n,N_n}(\t^\s) \right] \\
    & \ge \frac{|\t - \t^\s|}{\varepsilon} \left[
    v\left( \t^\s + \varepsilon \frac{\t - \t^\s}{|\t - \t^\s|} \right) -
    v(\t^\s) - \frac{2\delta}{3} \right] \ge \frac{\delta}{3}
  \end{align*}
  where the last two inequalities come from~\eqref{eq:luc}. If we apply this
  inequality for $\t = \t_{n, N_n}$, we obtain a contradiction since
  $v_{n,N_n}(\t_{n,N_n}) - v_{n, N_n}(\t^\s) \le 0$. Hence, we deduce that for
  all $n \ge n_\delta$, $|\t_{n,N_n} - \t^\s| < \varepsilon$. Therefore,
  $\t_{n,N_n}$ converges a.s. to $\t^\s$. If we combine this result with the
  local uniform convergence of $v_{n,N_n}$ to the continuous function $v$, we
  deduce that $v_{n,N_n}(\t_{n,N_n})$ converges a.s. to $v(\t^\s)$.

  Moreover, we get by Equation~\eqref{eq:upper_exp} that for all $K>0$
  \begin{align*}
    & \sup_{|\t| \le K} \abs{\partial_{\t^{(j)}} \psi(X_T)^2 \expp{-\t \cdot W_T +
    \frac{1}{2} |\t|^2 T}} \\
    & \quad \le \expp{K^2 T/2} \psi(X_T)^2 \left(K + (\expp{K W_t^{(j)}} +
    \expp{- K W_t^{(j)}})\right) \prod_{i=1}^q (\expp{K W_t^{(i)}} + \expp{- K
    W_t^{(i)}}). 
  \end{align*}
  The r.h.s is integrable by Condition~\ref{A:nonul}. Hence, $\EE\left[
  \sup_{|\t| \le K} \abs{\nabla_\t \psi(X_T)^2 \expp{-\t \cdot W_T + \frac{1}{2}
  |\t|^2 T}} \right] < +\infty$. Similarly, one can prove that $\EE\left[
  \sup_{|\t| \le K} \abs{\nabla_\t^2 \psi(X_T)^2 \expp{-\t \cdot W_T +
  \frac{1}{2} |\t|^2 T}} \right] < +\infty$. 
  Then, to prove the central limit theorem
  governing the convergence of $\t_{n, N_n}$ to $\t^\s$, we reproduce the proof
  of \cite[Theorem A2, pp. 74]{RubShapiro}, which is mainly based on the a.s. local
  uniform convergence of $\nabla v_{n, N_n}$ and on its asymptotic
  normality ensuing from Theorem~\ref{thm:lindeberg}.
\end{proof}

\subsection{A second stage Monte Carlo approach}\label{sec:mc}

In this section, we aim at building adaptive Monte Carlo estimators in the setting of
discretized diffusion processes following the spirit of \cite{JourLel}. Our setting
differs mainly because we want to let both the number of time steps and the number of samples
go to infinity. Asymptotic results rely on a uniform control of the triangular arrays
involved in the adaptive importance sampling Monte Carlo estimator.  The technical
results from Section~\ref{sec:slln} will be tremendously useful to provide such controls.

Using the estimators of $\t^\s$ studied in the previous section, we define
a Monte Carlo estimator of $\EE[\psi(X_T)]$ based on
Equation~\eqref{eq:is}. We introduce the $\sigma$-algebra $\Gc$
generated by the samples $(W_i)_{i \ge 1}$ used to compute $\t_n$ and
$\t_{n, N_n}$. 

Let $(\tilde W_i)_i$ be i.i.d. samples according to the law of $W$ but
independent of $\Gc$. Conditionally on $\Gc$, we introduce  i.i.d. samples
$(\tilde X_i(\t_{n, N_n}))_i$ following the law of $X(\t_{n, N_n})$ such that
for each $i$, $\tilde X_i(\t_{n, N_n})$ is the solution of the SDE driven by $\tilde W_{i}$.
We introduce $(\tilde\Gc_{k})_{k>0}$ the filtration defined  by
$\tilde\Gc_{k}=\sigma(\tilde W_{i}, 1\leq i\leq k)$ and $\Gc^\sharp_k = \Gc \vee \tilde
\Gc_k$. For each $i>0$, we
also consider $\tilde X_i^{n}(\t_{n, N_n})$ defined as the Euler discretization of
$\tilde X_i(\t_{n, N_n})$.  Based on these new sets of samples, we define
\begin{align*}
  M_{n, N_n} = \frac{1}{N_n} \sum_{i=1}^{N_n} g(\t_{n,N_n}, \tilde
  X^{n}_{T,i}(\t_{n, N_n}), \tilde W_{T,i}), 
\end{align*}
where the function $g : \RR^q \times \RR^d \times \RR^q \rightarrow \RR$ is defined by
\begin{equation}
  \label{eq:fct_g}
  g(\t, x, y) \defeq \psi(x) \expp{-\t \cdot y -\frac{1}{2} |\t|^2T}.
\end{equation}
For the clearness of the coming proofs, it is convenient to introduce the following
notation
\begin{align*}
  M_{n, N_n} (\t) = \frac{1}{N_n} \sum_{i=1}^{N_n} g(\t, \tilde
  X^{n}_{T,i}(\t), \tilde W_{T,i}).
\end{align*}
Note that $M_{n, N_n} = M_{n, N_n} (\t_{n, N_n})$.

\begin{theorem}
  \label{thm:slln} Assume that Assumption~\eqref{eq:Hbsigma} holds and that
  $\psi \in \Hc_\alpha$ for some $\alpha>0$. Then, $M_{n, N_n} \longrightarrow
  \EE[\psi(X_T)]$ a.s. when $n \rightarrow +\infty$.
\end{theorem}
\begin{proof}
  Using the conditional independence of the samples $(\tilde X_i^{n}(\t_{n,
  N_n}), \tilde W_i)_i$, we have 
  \begin{align*}
    \EE[ g( \t_{n,N_n}, \tilde X^{n}_{T,i}(\t_{n, N_n}), \tilde W_{T,i})
    | \Gc ] = \EE[\psi(X^n_T)] \defeq e_n \quad \mbox{for all $i>0$}.
  \end{align*}
  Let $\Vc \subset \RR^q$ be a compact neighbourhood of $\t^\s$. We define the
  sequence 
  $$
  Y_{i,n} = \left(g(\t_{n,N_n}, \tilde
    X^{n}_{T,i}(\t_{n, N_n}), \tilde W_{T,i}) - e_n \right)
  \ind{\t_{n,N_n} \in \Vc}
  $$
  and its empirical average $\overline Y_{m,n} = \inv{m} \sum_{i=1}^m Y_{i,n}$
  for all $m>0$. It is obvious that $\EE[Y_{i,n}] = 0$ and using the conditional
  independence $\EE[\abs{\overline Y_{m,n}}^2] = \frac{1}{m}
  \EE[\abs{Y_{1,n}}^2]$.
  \begin{align*}
    \EE[ \abs{Y_{1,n}}^2] & \le \EE\left[ \EE\left[ |g(\t_{n,N_n},
    \tilde X^{n}_{T,i}(\t_{n, N_n}), \tilde W_{T,i}) - e_n|^2 \Big| \Gc \right]
    \ind{\t_{n,N_n} \in \Vc}\right] \\
    & \le \EE\left[ v_n(\t_{n,N_n}) \ind{\t_{n,N_n} \in \Vc}\right] 
    \le \sup_{\t \in \Vc} v_n(\t).
  \end{align*}
  We know that $v_n$ is convex and converges point-wise to $v$, which is also
  convex and continuous. Hence, $v_n$ converges locally uniformly to $v$, which
  implies that for all compact sets $K\subset \RR^q$, $\lim_{n \rightarrow +\infty} \sup_{\t \in K} v_n(\t) = \sup_{\t
    \in K} v(\t)$. Hence, $\sup_n \sup_{\t \in \Vc} v_n(\t) < + \infty$. 
  Applying Proposition~\ref{prop:slln2} proves that
  ${\overline Y_{N_n,n} \xrightarrow[n \rightarrow +\infty]{a.s.} 0}$. As $\t_{n, N_n}$
  converges a.s. to $\t^\s \in K$, this also implies that 
  $\lim_{n \rightarrow +\infty} M_{n, N_n} = \EE[\psi(X_T)]$ a.s.
\end{proof}

\begin{theorem}
  \label{thm:clt} Under the assumptions of Theorem~\ref{thm:slln} and if
  Condition~\eqref{eq:Hen} holds, we have  $$\sqrt{N_n}
  (M_{n, N_n} - \EE[\psi(X_T)]) \Longrightarrow \Nc(C_{\psi}(T,\alpha),\sigma^2) \mbox{
  when } n\rightarrow +\infty.$$
where $\sigma^2=  \EE\Bigl[\psi( X_{T})^2 \expp{-\t^\s\cdot  W_{T}+\frac{1}{2}|\t^\s|^2T}\Bigr]-\EE[\psi(X_T)]^2.$
\end{theorem}
\begin{remark}
  \label{rem:pasfixe} Assume the number of time steps used in the Euler scheme is fixed to
  $n=1$ and consider the estimator $M_{1, N}(\t_{1, N})$. Then, we know from \cite[Theorem
  3.4]{Lel13} that, when $N \to \infty$,
  \begin{align*}
    & M_{1, N}(\t_{1, N}) \longrightarrow \EE[g(\t_1, X^1_T(\t_1), W_T)] \quad a.s. \\
    & \sqrt{N} (M_{1, N}(\t_{1, N}) - \EE[g(\t_1, X^1_T(\t_1), W_T)]) 
    \Longrightarrow  \Nc(0, \sigma_1^2)
  \end{align*}
  with $\sigma_1^2=  \EE\Bigl[\psi( X^1_{T})^2 \expp{-\t_1\cdot
    W_{T}+\frac{1}{2}|\t_1|^2T}\Bigr]-\EE[\psi(X^1_T)]^2.$
\end{remark}
\begin{proof}
  We can write the left hand side of the convergence result by introducing
  $M_{n,N_n}(\t^\s)$
  \begin{align*}
    \sqrt{N_n} (M_{n, N_n} - \EE[\psi(X_T)])  = \sqrt{N_n} (M_{n,
    N_n}(\t_{n,N_n}) - M_n(\t^\s)) + \sqrt{N_n} (M_{n, N_n}(\t^\s) - \EE[\psi(X_T)])  
  \end{align*}
  The convergence of the last term on the r.h.s $\sqrt{N_n} (M_{n, N_n}(\t^\s) -
  \EE[\psi(X_T)])$ is governed by the central limit theorem for Euler 
  Monte Carlo, which yields the announced limit (see \cite{DuffieGlynn95}).
  It remains to prove that
  $\sqrt{N_n} (M_{n, N_n}(\t_{n,N_n}) - M_{n,N_n}(\t^\s))$ converges to zero in
  probability.

  Let $\varepsilon >0$ and $\alpha < \frac{1}{2}$,
  \begin{align*}
    &  \PP\left( \sqrt{N_n}  |M_{n, N_n}(\t_{n,N_n}) - M_{n,N_n}(\t^\s)| > \varepsilon
    \right)   \\ 
    &= \quad  \PP\left( \sqrt{N_n} |M_{n, N_n}(\t_{n,N_n}) - M_{n,N_n}(\t^\s)| >
    \varepsilon \: ; \; N_n^\alpha |\t_{n,N_n} - \t^\s| > 1 \right) \\
    & \quad +\PP\left( \sqrt{N_n} |M_{n, N_n}(\t_{n,N_n}) - M_{n,N_n}(\t^\s)| > \varepsilon
    \: ; \; N_n^\alpha |\t_{n,N_n} - \t^\s| \le 1  \right)  \\
    &= \quad  \PP\left( \; N_n^\alpha |\t_{n,N_n} - \t^\s| > 1 \right) \\
    & \quad +\PP\left( \sqrt{N_n} |M_{n, N_n}(\t_{n,N_n}) - M_{n,N_n}(\t^\s)|
      \ind{N_n^\alpha |\t_{n,N_n} - \t^\s| \le 1} > \varepsilon \right).
  \end{align*}
  By Theorem~\ref{thm:theta_cv}, $\PP\left( \; N_n^\alpha |\t_{n,N_n} - \t^\s| >
  1 \right)$ tends to zero when $n$ goes to infinity. Let $K>0$ s.t.  for all
  $n$ large enough $\{\t \in \RR^q \: : \: |\t - \t^\s| \le N_n^{-\alpha} \}
  \subset B(0,K)$.  We can bound the second
  term on the r.h.s. by using Markov's inequality 
  \begin{align*}
    & \PP\left( \sqrt{N_n} |M_{n, N_n}(\t_{n,N_n}) - M_{n,N_n}(\t^\s)|
    \ind{N_n^\alpha |\t_{n,N_n} - \t^\s| \le 1} > \varepsilon \right)  \\
    & \le  \frac{N_n}{\varepsilon^2} \EE\left[ |M_{n, N_n}(\t_{n,N_n}) -
    M_{n,N_n}(\t^\s)|^2
    \ind{\t_{n,N_n}\in B(0,K)} \right] \\
    & \le \frac{1}{\varepsilon^2} \EE\left[ |g(\t_{n,N_n}, \tilde X_{T}^{n}(
      \t_{n,N_n}), \tilde W_T)  - g(\t^\s, \tilde X_{T}^{n}(\t^\s), \tilde W_T)|^2
    \ind{\t_{n,N_n} \in B(0,K)} \right] \\
    & \le \frac{1}{\varepsilon^2} \EE\left[ |g(\t_{n,N_n}, \tilde X_{T}^{n}
      (\t_{n,N_n}), \tilde W_T)  - g(\t_{n,N_n}, \tilde X_{T}(\t_{n,N_n}), \tilde W_T)|^2
    \ind{\t_{n,N_n} \in B(0,K)} \right]  \\
    & \quad + \frac{1}{\varepsilon^2} \EE\left[ |g(\t_{n,N_n}, \tilde X_{T}(\t_{n,N_n})
      \tilde W_T)  - g(\t^\s, \tilde X_{T}^{n}(\t^\s), \tilde W_T)|^2
    \ind{\t_{n,N_n} \in B(0,K)} \right].
  \end{align*}
  We treat each of the two terms separately.

  \noindent \textbf{$\blacktriangleright$ First term}

  From the independence between $\t_{n,N_n}$ and $\tilde W$, we can write
  \begin{align*}
    & \EE\left[ |g(\t_{n,N_n}, \tilde X_{T}^{n}(\t_{n,N_n}), \tilde W_T)  -
      g(\t_{n,N_n}, \tilde X_{T}(\t_{n,N_n}), \tilde W_T)|^2
    \ind{\t_{n,N_n} \in B(0,K)} \right]   \\ 
    & \quad = \EE\left[ |\psi(X_T^n) - \psi(X_T)|^2 \exp(-\t_{n,N_n} \cdot \tilde W_T +
    \frac{1}{2} |\t_{n,N_n}|^2 T)  \ind{\t_{n,N_n} \in B(0,K)} \right]  \\
    & \quad \le \EE\left[ |\psi(X_T^n) - \psi(X_T)|^{2(1+\eta)}
    \right]^{\frac{1}{1+\eta}}  \expp{\frac{1+2\eta}{2\eta} K^2 T},\quad \mbox{for some }\eta>0.
  \end{align*}
  Relying on the uniform integrability ensured by property \ref{eq:P} and
  since $\psi\in\Hc_{\alpha}$, we can let $n$ go to infinity inside the
  expectation to obtain that
  $$
  \lim_{n \rightarrow +\infty} \EE\left[ |g(\t_{n,N_n}, \tilde X_{T}^{n}
    (\t_{n,N_n}), \tilde W_T)  - g(\t_{n,N_n}, \tilde X_{T}(\t_{n,N_n}), \tilde W_T)|^2
    \ind{\t_{n,N_n} \in B(0,K)} \right]  = 0.
  $$

  \noindent \textbf{$\blacktriangleright$ Second term}

  Since the function $g$ is continuous w.r.t its first two parameters and
  $X_T^{\t}$ is continuous w.r.t the parameter $\t$, $\lim_{n \rightarrow
  +\infty} g(\t_{n,N_n}, \tilde X_{T}(\t_{n,N_n}), \tilde W_T)  - g(\t^\s,
  \tilde X_{T}^n(\t^\s), \tilde W_T) = 0$ a.s. To conclude the proof, we need
  to show that the family of r.v. 
  $$
  \left(|g(\t_{n,N_n}, \tilde X_{T}(
    \t_{n,N_n}), \tilde W_T)  - g(\t^\s, \tilde X_{T}^{n}(\t^\s), \tilde W_T)|^2
  \ind{\t_{n,N_n} \in B(0,K)}\right)_n
  $$
  is uniformly integrable. 
  
  First, for any $\t \in \RR^q$ and $2(1+\eta)>a>2$
  \begin{align}
    \label{eq:tcl-rem-1}
    \EE\left[|g(\t, \tilde X_{T}(\t), \tilde W_T)|^a\right] & =
    \EE \left[ |\psi(\tilde X_{T})|^a \expp{- (a-1) \t\cdot \tilde W_T +
    \frac{(a-1) |\t|^2 T}{2} }\right] \nonumber \\
    & \le \EE \left[ |\psi(\tilde X_{T})|^{2(1+\eta)}
    \right]^{\frac{2(1+\eta)}{a}} \expp{C |\t|^2}
  \end{align}
  where $C$ is a constant only depending on $a$ and $T$.
  This yields that for some $\delta  >0$ and some constant $C>0$
  independent of $\t$, $ \EE\left[ |g(\t, \tilde X_{T}(\t), \tilde W_T)|^{2+\delta}
  \right] < C \expp{C |\t|^2}$.  Then, we get
 \begin{align*}
  & \sup_n \EE\left[|g(\t_{n,N_n}, \tilde X_{T}(\t_{n,N_n}), \tilde
   W_T)|^{2+\delta}\ind{\t_{n,N_n} \in B(0,K)} \right] \\
   & =
   \sup_n \EE\left[\EE\left[|g(\t_{n,N_n}, \tilde X_{T}(\t_{n,N_n}), \tilde W_T)|^{2+\delta}
   | \t_{n,N_n} \right]\ind{\t_{n,N_n} \in B(0,K)}  \right] \\
   & \le \sup_n  C \EE\left[ \expp{C |\t_{n,N_n}|^2} \ind{\t_{n,N_n} \in B(0,K)}
   \right] \le C \expp{C K}.
  \end{align*}
  We can similarly obtain that
  \begin{align*}
    & \sup_n \EE\left[|g(\t^\s, \tilde X_{T}^{n}(\t^\s), \tilde
      W_T)|^{2+\delta}\right] \le \sup_n \EE\left[ |\psi(X_T^n)|^{2(1+\eta)}
    \right]^{\frac{2(1+\eta)}{2+\delta}} \expp{C|\t^\s|^2}.
  \end{align*}
  This proves that the family of r.v. $$\left(|g(\t_{n,N_n}, \tilde X_{T}
    (\t_{n,N_n}), \tilde W_T) - g(\t^\s, \tilde X_{T}^{n}(\t^\s), \tilde W_T)|^2
    \ind{\t_{n,N_n} \in B(0,K)}\right)_n$$ is uniformly integrable, which ends the
  proof.
\end{proof}

\section{Multilevel Importance sampling Monte Carlo}\label{sec:mul}

In the recent years, many works showed that MLMC supersedes Monte Carlo when combined with
discretization schemes. Then, it has become natural to investigate how this new approach
could be coupled with existing variance reduction techniques and in particular with
importance sampling.  In this section, we study the mathematical properties of our
importance sampling MLMC estimator $Q_L(\hat \lambda_0, \dots, \hat \lambda_L)$. First, we
start by proving the existence and uniqueness of $\hat \lambda_0, \dots, \hat \lambda_L$
in Section~\ref{sec:cv-mltheta} and then we prove a strong law of large numbers and a
central limit theorem for $Q_L(\hat \lambda_0, \dots, \hat \lambda_L)$ in
Section~\ref{sec:ml-cv}.

\subsection{General framework}
\label{sec:ml-framework}

Our multilevel importance sampling estimator writes
\begin{align}
  \label{eq:estimator}
  Q_L(\la_0,\dots,\la_L) & =\frac{1}{N_0}\sum_{k=1}^{N_0}
  \psi( \tX^{m^0}_{T,0,k}(\la_0)) \Ec^-(\tilde W_{0,k}, \la_0) \nonumber \\
  & \qquad +\sum_{\ell=1}^{L}\frac{1}{N_\ell}
  \sum_{k=1}^{N_\ell}\left(\psi( \tX^{m^{\ell}}_{T,\ell,k}(\la_{\ell}))-\psi(
  \tX^{m^{\ell-1}}_{T,\ell,k}(\la_{\ell}))\right) \Ec^-(\Tilde W_{\ell,k}, \la_\ell).
\end{align}
For any fixed $\l \in\{1,\cdots,L\}$, the random variables $(\tilde W_{\l,k})_{1
  \le k \le N_\l}$ are independent and are distributed according to the Brownian
law. We assume that for $\l, \l' \in\{1,\cdots,L\}$, with $\l \ne \l'$, the
blocks $(\tilde W_{\l,k})_{1 \le k \le N_\l}$ and $(\tilde W_{\l',k})_{1 \le k
  \le N_{\l'}}$ are independent.  For any fixed $\l \in\{1,\cdots,L\}$ and $k \in
\{1, \dots, N_\l\}$, the variables $\tX^{m^{\ell}}_{T,\l,k}(\lambda_\l)$ (resp.
$\tX^{m^{\ell-1}}_{T,\l,k}(\lambda_\l)$) are the terminal values of the Euler
schemes of $X(\lambda_\l)$ with $m^\l$ (resp. $m^{\l-1}$) time steps built using
the same Brownian path $\tilde W_{\l, k}$. The key of the multilevel approach
is to use the same Brownian path to compute $\tX^{m^{\ell}}_{T,\l,k}(\lambda_\l)$ 
and $\tX^{m^{\ell-1}}_{T,\l,k}(\lambda_\l)$.  The blocks of random variables used in
two different levels are independent. From these assumptions, one can compute
the variance of the multilevel estimator given by 
$$
\Var[Q_L]= N_0^{-1} \sigma_0(\la_0)^2 +\sum_{\ell=1}^L N^{-1}_{\ell} \frac{(m-1)T}{m^\l}
\sigma_{\ell}^2(\lambda_{\ell})
$$
where 
\begin{align*}
  \sigma_0^2 (\la_0)  & \defeq \Var[\psi(X^{m^0}_{T}(\lambda_0)) \Ec^-(W,\lambda_0)] \\
  \sigma_{\ell}^2(\lambda_\l) & \defeq \frac{m^\l}{(m-1) T} \Var\left[\left\{\psi(
      X^{m^{\ell}}_{T}(\lambda_\l)) - \psi(
  X^{m^{\ell-1}}_{T}(\lambda_\l))\right\}\Ec^-(W, \la_\l) \right].
\end{align*}
By applying~\eqref{eq:is}, the variances of each level $\l \ge 0$ can be written
$\sigma^2_\l(\la_\l) = v_\l(\lambda_\l) - \Xi_\l^2$
with
\begin{align}
  \label{eq:v_xi_0}
  v_0(\lambda_0) &\defeq \EE\left[ \psi( X^{m^{0}}_{T})^2 
  \Ec^+(W, \la_0) \right], \quad
  \Xi_0  \defeq \EE\left[\psi( X^{m^0}_{T}) \right] \\
  \label{eq:vl}
  v_\l(\lambda_\l) &\defeq \frac{m^\l}{(m-1)T} \EE\left[ \abs{\psi( X^{m^{\ell}}_{T}) - \psi(
  X^{m^{\ell-1}}_{T})}^2 \Ec^+(W, \la_\l) \right], \\
  \label{eq:xi_l}
  \Xi_\l  &\defeq \sqrt{\frac{m^\l}{(m-1)T}} \EE\left[\psi( X^{m^{\ell}}_{T}) - \psi(
  X^{m^{\ell-1}}_{T})\right]
\end{align}
and $ \Ec^+(W, \lambda) \defeq \expp{-\lambda \cdot W_{T} + \inv{2} \abs{\lambda}^2 T}.$
Hence, the global variance is given by
$$
\Var[Q_L]=N^{-1}_{0} (v_0(\la_0) - \Xi_0^2) +\sum_{\ell=1}^L N^{-1}_{\ell}  \frac{(m-1)T}{m^\l}\left(
  v_{\ell}(\lambda_{\ell}) -  \Xi_\l^2  \right).
$$
To actually minimize the functions $\la \longmapsto
v_{\ell}^2(\lambda)$, we consider the sample average approximation of $v_\l$ with
$N_\l'$ samples 
\begin{align*}
  v_{0, N_0'}(\lambda_0) &\defeq \frac{1}{N_0'} \sum_{k=1}^{N_0'} \psi(
  X^{m^0}_{T,0, k})^2 \Ec^+(W_{0,k}, \la_0), \\
  v_{\l, N_\l'}(\lambda_\l) &\defeq \frac{1}{N_\l'} \sum_{k=1}^{N_\l'} \frac{m^\l}{(m-1)T} \abs{\psi(
  X^{m^{\ell}}_{T,\l, k}) - \psi( X^{m^{\ell-1}}_{T,\l, k})}^2 \Ec^+(W_{\l,k}, \la_\l).
\end{align*}

\subsection{Convergence of the importance sampling parameters}
\label{sec:cv-mltheta}

From Lemma~\ref{lem:convex}, we deduce that $v_{\l, N_\l'}$ has a unique minimum
\begin{align*}
  \widehat \la_{\l} = \arg\min_{\la \in \RR^q}  v_{\l, N_\l'}(\la).
\end{align*}

\begin{theorem}
  \label{thm:cv_var} Assume $b$ and $\sigma$ are $C^1$ with bounded
  derivatives, $\psi\in\Hc_\alpha$ for some $\alpha \ge 1$, $\psi$ is $C^1$ and $\nabla\psi$ has
  polynomial growth. Then, the sequence of random functions $(v_{\l,N_\l'} :
  \lambda \in \RR^q \rightarrow v_{\l, N_{\l}'} (\lambda))_\l$  converges a.s.
  locally uniformly to the strongly convex function $v: \RR^q \rightarrow \RR$
  defined by 
  \begin{equation}
    \label{eq:variance-ml}
    v(\lambda)\defeq \EE\left[ \left(
        \nabla \psi(X_T) \cdot U_T\right)^2 \Ec^+(W, \lambda) \right]
  \end{equation}
  with 
  \begin{align}
    \label{eq_U_theta}
    dU_t=\nabla b(X_t) U_t dt+\sum_{j=1}^q\nabla{\sigma}_j(X_t)U_t
    dW^j_t-\frac{1}{\sqrt{2}}\sum_{ij,=1}^q\nabla{\sigma}_j(X_t){\sigma}_{i}(X_t)d
    \check W^{i,j}_t
  \end{align}
  where $\check W$ is a Brownian motion independent of
  $W$ with values in $\RR^{q \times q}$.

  Moreover,
  $\widehat \lambda_{\l}$ converges a.s. to $\la^\s \defeq \arg\min_\lambda
  v(\lambda)$, when $\l \to +\infty$.
\end{theorem}
\begin{proof}
  Let us define the doubly indexed sequence
  \begin{align*}
    Y_{k,\l}(\lambda) = &  \frac{m^\l}{(m-1)T} \abs{\psi(
      X^{m^{\ell}}_{T,k}) - \psi( X^{m^{\ell-1}}_{T,k})}^2 
    \Ec^+(W_k, \lambda).
  \end{align*}
  For any fixed $\l$, the sequence $(Y_{k,\l}(\lambda))_k$ is i.i.d. so that for
  any $k$, $\EE[ Y_{k,\l}(\lambda)] = y_\l(\lambda)$ with
  \begin{align*}
    y_{\l}(\lambda) = & \EE\left[ \frac{m^\l}{(m-1)T} \abs{\psi(
        X^{m^{\ell}}_{T}) - \psi( X^{m^{\ell-1}}_{T})}^2 \Ec^+(W, \lambda) \right].
  \end{align*}
  We deduce from Proposition~\ref{prop:conv-moments-level} that the sequence $(y_\l)_\l$
  converges pointwise to the continuous function $\EE\left[ \left(\nabla \psi(X_T) \cdot
      U_T\right)^2 \Ec^+(W, \lambda) \right]$, thus satisfying
      Assumption~\ref{f_m-pointwise}.  The i.i.d. property of  the sequence
  $(Y_{k,\l}(\lambda))_k$ also implies that
  \begin{align}
    \label{eq:A1-ml}
    \EE\left[ \sup_{\abs{\lambda}\le K} \inv{N} \left(\sum_{k=1}^{N}
          Y_{k,\l}(\lambda)\right)^2 \right] \le
    \EE\left[ \inv{N} \sum_{k=1}^{N}
         \sup_{\abs{\lambda}\le K} Y_{k,\l}(\lambda)^2 \right] \le
       \frac{1}{N} \EE\left[ \sup_{\abs{\lambda}\le K}Y_{1,\l}(\lambda)^2 \right].
  \end{align}
  \begin{align}
    \label{eq:A1-sup-ml}
    \EE\left[ \sup_{\abs{\lambda}\le K}Y_{1,\l}(\lambda)^2 \right]^2      \le
    \EE\left[ \left(\frac{m^\l}{(m-1)T} \abs{\psi( X^{m^{\ell}}_{T}) - \psi(
          X^{m^{\ell-1}}_{T})}^2 \right)^4 \right] \EE\left[ \sup_{\abs{\lambda}
        \le K} \Ec^+(W, \lambda)^4 \right].
  \end{align}
  Using the following upper bound
  \begin{align}
    \label{eq:upper_exp}
    \sup_{|\lambda| \le K} e^{-\lambda \cdot W_{T}+\frac{1}{2}|\lambda|^{2}T} \le
    \expp{\frac{1}{2} K^{2}T} \prod_{l=1}^q (\expp{K W_T^{(l)}} + \expp{-K
    W_T^{(l)}}),
  \end{align}
   $\EE\left[ \sup_{\abs{\lambda} \le K}
    \Ec^+(W, \lambda)^4 \right] < +\infty$.   
  Let us have a closer look at the first term in~\eqref{eq:A1-sup-ml}. From Condition~\eqref{eq:Halpha},
  we can write
  \begin{align*}
    \EE\left[ \left(m^\l \abs{\psi( X^{m^{\ell}}_{T}) - \psi(
          X^{m^{\ell-1}}_{T})}^2\right)^4 \right] \le 
    C \EE\left[ m^{4\l} \abs{X^{m^{\ell}}_{T} - X^{m^{\ell-1}}_{T}}^{8
        \alpha} \left(1 + \abs{X^{m^{\ell}}_{T}}^{8 \beta} +
        \abs{X^{m^{\ell-1}}_{T}}^{8 \beta}\right) \right].
  \end{align*}
  By using the strong rate of convergence of the Euler scheme, we
  notice that for any $p>1$, 
  \begin{align*}
   \EE\left[ m^{4\l p} \abs{X^{m^{\ell}}_{T} - X^{m^{\ell-1}}_{T}}^{8
       \alpha p} \right] \le m^{4\l p} C \left(  m^{-4 \alpha p \ell} +
     m^{-4 \alpha p (\ell-1) }\right) \le C m^{4 \alpha p - 4 \l p (\alpha-1)}.
  \end{align*}
  Hence, since $\alpha \ge 1$, by using the Cauchy Schwartz inequality we easily
  check that
  \begin{align*}
    \sup_{\l} \EE\left[ \left(\frac{m^\l}{(m-1)T} \abs{\psi( X^{m^{\ell}}_{T}) - \psi(
          X^{m^{\ell-1}}_{T})}^2\right)^4 \right] < +\infty.
  \end{align*}
  By combining all these results into~\eqref{eq:A1-sup-ml}, we obtain that $\sup_{\l}  \EE\left[
    \sup_{\abs{\lambda} \le K} Y_{1,\l}^2(\lambda) \right] < +\infty$. Then, we deduce along with~\eqref{eq:A1-ml}
  that the sequence $(Y_{k,\l})_{k, \l}$ satisfies Assumption~\ref{slln2-u}
  of Proposition~\ref{prop:slln2-u}. \\

  Let $\delta>0$ and $\lambda \in \RR^d$.
  \begin{align*}
   & \EE\left[ \sup_{\abs{\mu - \lambda} \le \delta} \abs{Y_{1, \l}(\lambda)
        - Y_{1,\l}(\mu)}\right]^2 \le  \\
   &\EE\left[ \left(\frac{m^\l}{(m-1)T} \abs{\psi(
        X^{m^{\ell}}_{T}) - \psi( X^{m^{\ell-1}}_{T})}^2 
      \right)^2 \right]
  \EE\left[\sup_{\abs{\mu - \lambda} \le \delta} \abs{\Ec^+(W, \lambda) -
      \Ec^+(W, \mu)}^2  \right].
  \end{align*}
  We have just proved that the first expectation on the r.h.s is bounded
  uniformly in $\l$. Since the exponential weights are a.s. continuous with
  respect to $\lambda$, it is clear that \\
  ${\lim_{\delta \to 0} \sup_{\abs{\mu -
        \lambda} \le \delta} \abs{\Ec^+(W, \lambda) - \Ec^+(W, \mu)}^2 = 0}$ a.s.
  Moreover, we can apply Lebesgue's theorem with the upper--bound given
  by~\eqref{eq:upper_exp} to deduce that 
  \begin{align*}
    \lim_{\delta \to 0} \sup_\l \EE\left[ \sup_{\abs{\mu - \lambda} \le \delta} \abs{Y_{1, \l}(\lambda)
        - Y_{k,\l}(\mu)}\right] = 0.
  \end{align*}
  Thus, Assumption~\ref{slln2-sup-u} of Proposition~\ref{prop:slln2-u} is satisfied.
  Finally, we can apply Proposition~\ref{prop:slln2-u} to prove that the sequence
  $\inv{N_\l'} \sum_{k=1}^{N_\l'} Y_{k,\l}$ converges a.s locally uniformly to $0$.  The
  convergence of $\widehat{\lambda_\l}$ to $\la^\s$ can be deduced by closely mimicking
  the proof of Theorem~\ref{thm:theta_cv}.
\end{proof}

\subsection{Strong law of large numbers and central limit theorem}
\label{sec:ml-cv}

Let us  introduce a sequence $(a_\ell)_{\ell \in \NN}$ of positive real numbers
such that $\lim_{L\rightarrow \infty}\sum_{\ell=1}^{L}a_\ell=\infty$. We
assume that the  sample size $N_{\ell}$ has the following form
\begin{equation}
  \label{eq-size}
  N_{\ell, L}^\rho=\frac{\rho(L)}{m^\ell a_\ell}\sum_{k=1}^{L}a_k,
  \quad \ell\in\{0,\cdots,L\}
\end{equation}
for some increasing function $\rho : \NN \rightarrow \RR$.

We choose this form  for $N_{\ell}$ because it is a generic form allowing us a
straightforward  use of the Toeplitz Lemma, which is a key tool to prove
the central limit theorem. Since $\lim_{L\rightarrow
  \infty}\sum_{\ell=1}^{L}a_\ell=\infty$,
for any sequence $(x_{\ell})_{\ell\geq 1}$ converging to some limit $x\in\RR$,
$$\lim_{L\rightarrow +\infty}
\frac{\sum_{\ell=1}^{L} a_{\ell} x_{\ell}}{\sum_{\ell=1}^{L}a_\ell} =x.$$

We define the $\sigma$-algebra $\Gc$ generated by the samples $(W_{\l,k})_{\l, k
  \ge 1}$ used to compute $\widehat \lambda_L$. In the above framework, the
variables $(\tilde W_{\l,k})_{\l,k}$ are independent of $\Gc$.  We also
introduce the filtration $(\tilde\Gc_{\l})_{\l>0}$ generated by $(\tilde W_{\l,k}, k \ge
1)_{\l}$ and the filtration $(\Gc^\sharp_{\l})_{\l>0}$ defined  as
$\Gc^\sharp_{\l}=\Gc\vee\tilde\Gc_\l$.

\begin{theorem}\label{thm:slln-ml}
  Assume that $\sup_L \sup_\l \frac{L^2 a_\l}{\rho(L) \sum_{k=1}^L a_k} <
  +\infty$.  Then, under the assumptions of Theorem \ref{thm:cv_var},
  $Q_L(\widehat \la_{0},\dots,\widehat \la_{L}) \longrightarrow \EE[\psi(X_T)]$
  a.s. when $L \to +\infty$.
\end{theorem}
For the choice $a_\l=1$ for all $\l$, the condition on $\rho$ reduces to $\sup_L
\frac{L}{\rho(L)} < +\infty$. 

\begin{proof}
  As $\EE[\psi(X^L_T)]$ converges to $\EE[\psi(X_T)]$ as $L$ goes to infinity,
  it is enough to show that $Q_L(\widehat \la_{0},\dots,\widehat \la_{L}) -
  \EE[\psi(X^L_T)]$ tends to $0$. 
  \begin{align}
    \label{eq:expanded-QL}
  Q_L(\widehat \la_0,\dots,\widehat \la_L)  - \EE[\psi(X^L_T)] & =
    \frac{1}{N^\rho_{0,L}}\sum_{k=1}^{N^\rho_{0,L}} 
    \psi( \tX^{m^0}_{T,0,k}(\widehat \la_0)) \Ec^-(\tilde W_{0,k}, \widehat \la_0) -
    \EE[\psi(X^{m^0}_{T, 0})] \nonumber\\
    & +\sum_{\ell=1}^{L}\frac{1}{N^\rho_{\ell,L}}
    \Bigg(\sum_{k=1}^{N^\rho_{\ell,L}}\left(\psi( \tX^{m^{\ell}}_{T,\ell,k}(\widehat \la_{\ell}))-\psi(
      \tX^{m^{\ell-1}}_{T,\ell,k}(\widehat \la_{\ell}))\right)
    \Ec^-(\Tilde W_{\ell,k}, \widehat \la_\ell) \nonumber\\
    & \phantom{+\sum_{\ell=1}^{L}\frac{1}{N_\ell} \Bigg(} 
    -  \EE\left[ \psi(\tX^{m^{\l}}_{T,\l}) -\psi( \tX^{m^{\l-1}}_{T,\l}) \right] \Bigg).
  \end{align}
  From Theorem~\ref{thm:slln} and Remark~\ref{rem:pasfixe}, we know that 
  \begin{align*}
    \frac{1}{N^\rho_{0,L}}\sum_{k=1}^{N^\rho_{0,L}} \psi( \tX^{m^0}_{T,0,k}(\widehat \la_0))
   \Ec^-(\tilde W_{0,k}, \widehat \la_0) - \EE[\psi(X^{m^0}_{T, 0}) ] \xrightarrow[L \to +\infty]{a.s.} 0.
  \end{align*}
  Then, it suffices to prove that the remaining terms in~\eqref{eq:expanded-QL}
  tend to $0$ with $L$. Let $\Vc$ be a compact neighbourhood of $\la^\s$.
  \begin{align*}
    & \sum_{\ell=1}^{L}\frac{1}{N^\rho_{\ell,L}}
    \Bigg(\sum_{k=1}^{N^\rho_{\ell,L}}\left(\psi( \tX^{m^{\ell}}_{T,\ell,k}(\widehat \la_{\ell}))-\psi(
      \tX^{m^{\ell-1}}_{T,\ell,k}(\widehat \la_{\ell}))\right)
    \Ec^-(\Tilde W_{\ell,k}, \widehat \la_\ell) -
    \EE\left[ \psi(\tX^{m^{\l}}_{T,\l}) -\psi( \tX^{m^{\l-1}}_{T,\l}) \right] \Bigg) =
    \\
    & \sum_{\ell=1}^{L}\frac{1}{N^\rho_{\ell,L}}
    \Bigg(\sum_{k=1}^{N^\rho_{\ell,L}}\left(\psi( \tX^{m^{\ell}}_{T,\ell,k}(\widehat \la_{\ell}))-\psi(
      \tX^{m^{\ell-1}}_{T,\ell,k}(\widehat \la_{\ell}))\right)
    \Ec^-(\Tilde W_{\ell,k}, \widehat \la_\ell) -
    \EE\left[ \psi(\tX^{m^{\l}}_{T,\l}) -\psi( \tX^{m^{\l-1}}_{T,\l}) \right] \Bigg)
    \ind{\widehat \lambda_\l \in \Vc} \\
    & + \sum_{\ell=1}^{L}\frac{1}{N^\rho_{\ell,L}}
    \Bigg(\sum_{k=1}^{N^\rho_{\ell,L}}\left(\psi( \tX^{m^{\ell}}_{T,\ell,k}(\widehat \la_{\ell}))-\psi(
      \tX^{m^{\ell-1}}_{T,\ell,k}(\widehat \la_{\ell}))\right)
    \Ec^-(\Tilde W_{\ell,k}, \widehat \la_\ell) -
    \EE\left[ \psi(\tX^{m^{\l}}_{T,\l} -\psi( \tX^{m^{\l-1}}_{T,\l}) \right] \Bigg)
    \ind{\widehat \lambda_\l \notin \Vc}
  \end{align*}
  For $\l$ large enough (although random),  $\ind{\widehat \lambda_\l \notin \Vc} = 0$.
  Hence, the second term in the above equation tends to $0$ a.s. when $L$ goes to infinity. It
  remains to prove that the first term also converges to zero.  To do so, we apply
  Proposition~\ref{prop:slln2} to the sequence 
  \begin{align*}
    Y_{\l,q} = & q \frac{1}{N_{\l, q}^{\rho}} \Bigg(\sum_{k=1}^{N_{\l,q}^\rho}\left(\psi(
      \tX^{m^{\l}}_{T,\l,k}(\widehat \la_{\l}))-\psi( \tX^{m^{\l-1}}_{T,\ell,k}(\widehat
      \la_{\l}))\right) \Ec^-(\Tilde W_{\l,k}, \widehat \la_\l) \\
    & \qquad - \EE\left[ 
      \left(\psi(\tX^{m^{\l}}_{T,\l})-\psi(\tX^{m^{\l-1}}_{T,\l})\right) \right]\Bigg)
    \ind{\widehat \lambda_\l \in \Vc} 
  \end{align*}
  and set $\overline Y_{L,q} = \inv{L} \sum_{\l=1}^L Y_{\l,q}$. Note that $\EE[Y_{\l, q}]
  = 0$ for all $\l$ and $q$.  Since the samples used in the different levels are
  independent and the $\hat \lambda_\l$'s are independent of the filtration $\tilde \Gc$,
  we can write
  \begin{align}
    \label{eq:lfgn1}
    \EE\left[\abs{\overline Y_{L,q}}^2\right] & = \inv{L^2} \EE\left[ \EE\left[
        \abs{\sum_{\l=1}^L Y_{\l,q}}^2 \Big| \Gc \right] \right]  
    = \inv{L^2} \sum_{\l=1}^L \EE\left[ \abs{Y_{\l,q}}^2 \right] .
  \end{align}
  Using the same kind of arguments, we obtain
  \begin{align*}
    \EE\left[ \abs{Y_{\l,q}}^2 \right] & \le  q^2 \frac{1}{N_{\l, q}^{\rho}}
    \EE\left[\left(\psi( \tX^{m^{\l}}_{T,\l})-\psi(
        \tX^{m^{\l-1}}_{T,\ell})\right)^2 
      \Ec^+(\Tilde W_{\l},\widehat  \la_\l)\ind{\widehat \lambda_\l \in \Vc}\right] \\
    &\le \frac{q^2 a_\l}{\rho(q) \sum_{k=1}^q a_k} \left\{m^\l
      \EE\left[\left(\psi( \tX^{m^{\l}}_{T,\l})-\psi(
          \tX^{m^{\l-1}}_{T,\ell})\right)^2 
        \Ec^+(\Tilde W_{\l},\widehat  \la_\l)\ind{\widehat \lambda_\l \in \Vc}\right]
    \right\}.
  \end{align*}
  From Proposition~\ref{prop:conv-moments-level},  the term into braces
  converges when $\l$ goes to infinity. Hence, using the assumptions on the
  function $\rho$, we get
  \begin{align}
    \label{eq:lfgn2}
    \sup_q \sup_\l \EE\left[ \abs{Y_{\l,q}}^2 \right] < +\infty.
  \end{align}
  By combining Equations~\eqref{eq:lfgn1} and~\eqref{eq:lfgn2}, we get that
  $\sup_L \sup_q L \EE\left[\abs{\overline Y_{L,q}}^2\right] < +\infty$. Hence,
  Proposition~\ref{prop:slln2} yields that $\overline Y_{L,L}$ vanishes when $L$
  goes to infinity and this ends the proof.
\end{proof}

\begin{theorem}
  \label{thm:clt-ml} Suppose that the assumptions of Theorem \ref{thm:cv_var}
  hold and that Condition~\eqref{eq:Hen} is satisfied. If
  $N_{\l, L}^\rho$ is given by~\eqref{eq-size} with 
  $\rho(L) = m^{2 \gamma L}(m-1) T$ and the sequence $(a_\l)_\l$ satisfies
  \begin{equation}
    \label{eq:rate_al}
    \lim_{L\rightarrow \infty}
    \frac{1}{\left(\sum_{\ell=1}^{L}a_\ell\right)^{p/2}}
    \sum_{\ell=1}^{L}a_\ell^{p/2} =0, \mbox{ for } p>2,
  \end{equation}
  then $m^{\gamma L} ( Q_L(\widehat \la_{0},\dots,\widehat \la_{L}) - \EE[\psi(X_T)])
  \Longrightarrow \Nc(C_{\psi}(T,\gamma),v(\la^\s))$  when $L \to \infty.$
\end{theorem}
The convergence rate does not depend on the number of samples $N_\l'$ provided that they
tend to infinity with $\l$.
\begin{proof}
  By assumption~\eqref{eq:Hen}, we have that $\lim_{L\rightarrow +\infty}m^{\gamma
  L}(\EE[\psi(X^{m^{L}}_T)-\psi(X_T)]=C_{\psi}(T,\gamma).$ The convergence of the level
    $0$ is governed by Theorem~\ref{thm:clt} (see Remark~\ref{rem:pasfixe}) which yields
    that, when $L \to \infty$, 
  $$
  \left( \frac{1}{\sqrt{N^\rho_{0, L}}}\sum_{k=1}^{N^\rho_{0, L}} \psi(\tilde
    X^{m^0}_{T,0,k}(\widehat \la_{0})) \Ec^-(\tilde W_{0,k}, \widehat \la_{0})
    -\EE[\psi(X^{m^0}_T)] \right) \Longrightarrow
    \Nc(0,\sigma_0^2(\hat \la_0)).
  $$
  Then, we deduce from the choice of the
  function $\rho$ that
  $$
  m^{\gamma L}\left( \frac{1}{N^\rho_{0, L}}\sum_{k=1}^{N^\rho_{0, L}} \psi(\tilde
    X^{m^0}_{T,0,k}(\widehat \la_{0})) \Ec^-(\tilde W_{0,k}, \widehat \la_{0})
    -\EE[\psi(X^{m^0}_T)] \right) \xrightarrow[L \to +\infty]{\quad \PP\quad }
  0.
  $$
  Since all the blocks are independent, it is sufficient to prove that 
  $$ 
  m^{\gamma L}\left(\sum_{\ell=1}^{L}\frac{1}{N^\rho_{\ell, L}}
    \sum_{k=1}^{N^\rho_{\ell,L}}\left(\psi(\tilde X^{m^{\ell}}_{T,\ell,k}(\widehat \la_{\ell}))-\psi(
      \tilde X^{m^{\ell-1}}_{T,\ell,k}(\widehat \la_{\ell}))\right)
    \Ec^-(\tilde W_{\ell,k}, \widehat \la_{\ell}) - \EE[\psi(X^n_T)]\right) \Longrightarrow \Nc(0, v(\la^\s)).
  $$
  To do so, we introduce the $(\Gc^\sharp_{l})_{l\geq 1}$-martingale
  array $(Y_l^{n})_{l\geq 1 }$ defined by
  $$ 
  Y_l^{n}\defeq \sum_{\ell=1}^{l}
  \frac{m^{\gamma L}}{N^\rho_{\ell,L}}\sum_{i=1}^{N^\rho_{\ell,L}}\left[\left(\psi(\tilde
      X^{m^{\ell}}_{T,\ell,i}(\widehat \la_{\ell}))-\psi( \tilde
      X^{m^{\ell-1}}_{T,\ell,i}(\widehat \la_{\ell}))\right)
    \Ec^-(\tilde W_{\ell,i}, \widehat \la_{\ell}) 
    - \EE\left[\psi( \tilde X^{m^{\ell}}_{T}) - 
      \psi( \tilde X^{m^{\ell-1}}_{T})\right]\right],
  $$
  so $\EE[Y_l^n] = 0$ for all $l,n$. According to Theorem
  \ref{thm:lindeberg}, we need to study the asymptotic behaviors of the two
  quantities
  $$
  \langle Y^{n}\rangle_{L}=\sum_{\l=1}^{L} \mathbb{E} \left[|
    Y_{\l}^{n}-Y_{\l-1}^{n}|^{2} \Big|\Gc^\sharp_{\l-1} \right]
  \mbox{ and }
  \sum_{\l=1}^{L} \EE\left[| Y_{\l}^{n}-Y_{\l-1}^{n}|^{p}
    \big|\Gc^\sharp_{\l-1} \right], \;\mbox{ for }p>2 \mbox{ as } n\rightarrow \infty.
  $$
  Note that $\widehat \la_{\l}$ is $\Gc^\sharp_{\l-1}$--measurable and for any $\la\in \RR^q$ the
  variables $(\tilde X^{m^{\l}}_{T,\l,i}(\la),\tilde X^{m^{\l-1}}_{T,\l,i}(\la))_{1\leq
    i\leq N_l}$ are independent of $\Gc^\sharp_{\l-1}$, then using
  (\ref{eq-size}) with $\rho(L) =
  m^{2 \gamma L}(m-1) T$, we rewrite the first quantity as follows 
  \begin{align*}
   \langle Y^{n}\rangle_L & =
    \frac{1}{\sum_{\l=1}^{L}a_\l}
    \sum_{\l=1}^{L}a_\l \left[v_\l(\widehat \lambda_\l)- \Xi_\ell^2\right]
  \end{align*}
  with $v_\l$ defined by~\eqref{eq:vl} and $\Xi_\l$ defined by~\eqref{eq:xi_l}. Let $\Vc$
  be a compact neighbourhood of $\la^\s$. We can write
  \begin{align}
    \label{eq:bracket}
   \langle Y^{n}\rangle_L & =
    \frac{1}{\sum_{\l=1}^{L}a_\l}
    \sum_{\l=1}^{L}a_\l \left[v_\l(\widehat \lambda_\l)- \Xi_\ell^2\right] \ind{\widehat
      \lambda_\l \in \Vc}
    + \frac{1}{\sum_{\l=1}^{L}a_\l}
    \sum_{\l=1}^{L}a_\l \left[v_\l(\widehat \lambda_\l)- \Xi_\ell^2\right] \ind{\widehat
      \lambda_\l \notin \Vc}.
  \end{align}

  From Proposition~\ref{prop:conv-moments-level}, we know that $\Xi_\ell
  \longrightarrow\EE[\nabla \psi(X_T).U_T]=0$, where the last equality is a straightforward
  consequence of \cite[Proposition 2.1]{Keb}. From
  Proposition~\ref{prop:conv-moments-level}, we know that the sequence of fucntions $v_\l$
  converges pointwise to $v$ defined by~\eqref{eq:variance-ml}.  Moreover, we can easily
  prove that this convergence is locally uniform. Hence, by the convergence of $\widehat
  \lambda_\l$ to $\la^\s$ (see Theorem~\ref{thm:cv_var}), we deduce that $v_\l(\widehat
  \lambda_\l) \ind{\widehat \lambda_\l \in \Vc}$ converges to $v(\la^\s)$ when $\l \to
  +\infty$. Moreover, for $\l$ large enough (although random), $\ind{\widehat \lambda_\l
  \notin \Vc} = 0$. 
  
  Thus, we deduce from the Toeplitz lemma that $ \langle Y^{n}\rangle_L
  \longrightarrow v(\la^\s)$ a.s.  Using Burkholder's inequality and
  Jensen's inequalty together with the assumptions on $\psi$ and Property
  \ref{eq:P}, we obtain that for any $p>2$, there exists $C_p>0$ such that 
  $$ 
  \sum_{\l=1}^{L} \EE\left[| Y_{\l}^{n}-Y_{\l-1}^{n}|^{p}
    \big|\Gc^\sharp_{\l-1} \right]\leq \frac{C_p}{\left(\sum_{\ell=1}^{L}a_\ell\right)^{p/2}}
  \sum_{\ell=1}^{L}a_\ell^{p/2}\xrightarrow[L\rightarrow\infty]{} 0 
  $$
  where the convergence to zero is ensured by~\eqref{eq:rate_al}.
  Consequently, we can apply Theorem \ref{thm:lindeberg} to achieve the proof.
\end{proof}
\begin{remark}\label{rem:ci}
  As usual, one can rescale $ m^{\gamma L} ( Q_L(\widehat \la_{0},\dots,\widehat
  \la_{L}) - \EE[\psi(X_T)])$ by an estimator of $v(\la^\s)$ to obtain a central limit
  theorem with variance $1$. Thanks to Theorem~\ref{thm:cv_var}, we know that $v_{\l,
    N_\l}(\hat \lambda_\l)$ is a convergent estimator of $v(\la^\s)$ and we can easily
  deduce from the proof of Theorem~\ref{thm:clt-ml} that under its assumptions
  \begin{align*}
    m^{2\gamma L} \left\{ 
      \vphantom{
      \frac{1}{N^{\rho}_{0,L}} \left(\frac{1}{N^{\rho}_{0,L}}\sum_{k=1}^{N^{\rho}_{0,L}}
        (\psi(X^{m^0}_{T}) \Ec^+(W_{0,k}, \lambda_0))^2 - \left(
          \frac{1}{N^{\rho}_{0,L}}\sum_{k=1}^{N^{\rho}_{0,L}} \psi(X^{m^0}_{T})
          \Ec^+(W_{0,k}, \lambda_0) \right)^2 \right)
    } \right.
      &
      \frac{1}{N^{\rho}_{0,L}} \left(\frac{1}{N^{\rho}_{0,L}}\sum_{k=1}^{N^{\rho}_{0,L}}
        (\psi(\tilde X^{m^0}_{T}) \Ec^+(\tilde W_{0,k}, \lambda_0))^2 - \left(
          \frac{1}{N^{\rho}_{0,L}}\sum_{k=1}^{N^{\rho}_{0,L}} \psi(\tilde X^{m^0}_{T})
          \Ec^+(\tilde W_{0,k}, \lambda_0) \right)^2 \right) \\
      & \left.
      \vphantom{
      \frac{1}{N^{\rho}_{0,L}} \left(\frac{1}{N^{\rho}_{0,L}}\sum_{k=1}^{N^{\rho}_{0,L}}
        (\psi(X^{m^0}_{T}) \Ec^+(W_{0,k}, \lambda_0))^2 - \left(
          \frac{1}{N^{\rho}_{0,L}}\sum_{k=1}^{N^{\rho}_{0,L}} \psi(X^{m^0}_{T})
          \Ec^+(W_{0,k}, \lambda_0) \right)^2 \right)
      }
      +\sum_{\ell=1}^L N^{-1}_{\ell}  \frac{(m-1)T}{m^\l}\left(
        \tilde v_{\ell, N_\l}(\lambda_{\ell}) -  \tilde \Xi_{\l, N_\l}^2  \right) \right\}
    \xrightarrow[L \to +\infty]{} v(\la^\s).
  \end{align*}
  Note that the quantities $\tilde v_{\l, N_\l}$ and $\tilde \Xi_{\, N_\l}$ are defined as in
  \eqref{eq:vl} and~\eqref{eq:xi_l} but using the tilde sample paths
  $(\tilde X_{\l, k})$ and $(\tilde W_{\l, k})$.  The term into braces, which can be
  computed online during the multilevel Monte Carlo procedure, can be used to build
  confidence intervals. Any convergent estimator of $v(\la^\s)$ could of course be
  used, but this one has the advantage to correspond to the true variance of the
  multilevel Monte Carlo estimator for any finite number of levels $L$ and not only
  asymptotically.
\end{remark}

\section{Strong law of large numbers for doubly indexed sequences}\label{sec:slln}

In this section, we prove two corner stone results used in the convergence of
the multilevel approach. We tackle the convergence of empirical averages of
doubly indexed random sequences when both indices tend to infinity together.

\begin{proposition}
  \label{prop:slln2} Let $(X_{n,m})_{n, m}$ be a doubly indexed sequence of
  vector valued random variables such that for all $n$, $\EE[X_{n,m}] = x_m$
  with $\lim_{m \to +\infty} x_m = x$ . We
  define $\overline X_{n,m} = \inv{n} \sum_{i=1}^n X_{i,m}$. Assume that the two
  following assumptions are satisfied
  \begin{hypo}
    \label{slln2}
    \begin{subhypo}
    \item \label{varbar} $\sup_n  \sup_m  n \Var\left( \overline
      X_{n,m} \right) < +\infty$.
    \item \label{var} $\sup_n \sup_m \Var\left( X_{n,m}
      \right) < +\infty$.
    \end{subhypo}
  \end{hypo}
  Then, for all increasing functions $\rho: \NN \to \NN$, $\overline
  X_{n, \rho(n)} \longrightarrow x$ a.s. and in $\LL^2$ when $n \to \infty$.
\end{proposition}
From this proposition, one can easily deduce the following corollary by
extracting a bespoke subsequence
\begin{corollary}\label{cor:slln}
  Assume that $(X_{n,m})_{n, m}$ be a doubly indexed sequence of vector valued
  random variables satisfying the assumptions of Proposition~\ref{prop:slln2}.
  Then, for any strictly increasing function $\xi: \NN \to \NN$, $\overline
  X_{\xi(n),n} \longrightarrow x$ a.s. and in $\LL^2$ when $n \to \infty$.
\end{corollary}
\begin{proof}[Proof of Proposition~\ref{prop:slln2}]
  The proof of this result closely mimics the one of \cite[Theorem
  IV.1.1]{revuz-proba}.  We introduce the sequence $(Y_{i,m})_{i,m}$ defined by
  $Y_{i,m} = X_{i,m} - x_m$, which satisfies $\EE[Y_{i, m}] = 0$.  As $\lim_{m
    \to \infty} x_m = x$, it is sufficient to prove that $\overline Y_{n,
    \rho(n)} \longrightarrow 0$ a.s. 

    Condition~\ref{varbar} implies the $\LL^2$ convergence to $0$. We introduce
  the sequence $(Z_{n,m})_n$ defined by  $Z_{n,m} = \sup\{ \abs{\bar Y_{k, m}} \;
    : \; n^2 \le k < (n+1)^2\}$. Let $k$ be such that $n^2 \le k < (n+1)^2$,
  then
  \begin{align*}
    \abs{\bar Y_{k, m}} & \le n^{-2} \left( n^2 \abs{\bar Y_{n^2, m}} + \sum_{i=n^2+1}^k
      \abs{Y_{i,m}} \right), \\ 
    Z_{n,m} & \le \abs{\bar Y_{n^2, m}} + \frac{1}{n^2} \sum_{i=n^2+1}^{(n+1)^2}
      \abs{Y_{i,m}}.
  \end{align*}
  Then, 
  \begin{align*}
    \EE[Z_{n,m}^2] \le \EE[\bar Y_{n^2, m}^2] + \sum_{i=n^2+1}^{(n+1)^2}
    \left(\frac{\EE[\abs{Y_{i,m}}^2]}{n^4} + 2 \frac{\EE[ \abs{\bar Y_{n^2, m}}
      \abs{Y_{i,m}}]}{n^2} \right)+ 2 \sum_{i, j = n^2 +1; i \neq j}^{(n+1)^2}
    \frac{\EE[ \abs{Y_{j, m}} \abs{Y_{i,m}}]}{n^4}.
  \end{align*}
  Let $\kappa > 0$ denote the maximum of the upper bounds involved in
  Assumption~\ref{slln2}. Using the Cauchy Schwartz inequality, we get
  \begin{align*}
    \EE[Z_{n,m}^2] & \le \frac{\kappa}{n^2} + \frac{\kappa ((n+1)^2 - n^2)}{n^4}
    + 2 \frac{\kappa^2((n+1)^2 - n^2)}{n^3}
    + 2 \frac{\kappa^2 ((n+1)^2-n^2)^2}{n^4} \\
    & \le \frac{\kappa}{n^2} + \frac{\kappa (2 n + 1)}{n^4}
    + 2 \frac{\kappa^2 (2n+1)}{n^3}
    + 2 \frac{\kappa^2 (2n+1)^2}{n^4}.
  \end{align*}
  Hence, for any function $\rho: \NN \to \NN$, $\EE[Z_{n,\rho(n)}^2] \le C
  n^{-2}$ where $C >0$ is a constant independent of $\rho$. Therefore, we have
  $\PP(Z_{n,\rho(n)} \ge n^{-1/4}) \le C n^{-3/2}$. This inequality implies 
  using the Borel Cantelli Lemma that, for $n$ large enough $Z_{n,\rho(n)} \le
  n^{-1/4}$ a.s. which yields the a.s. convergence to $0$.
\end{proof}

\begin{proposition}
  \label{prop:slln2-u} Let $(F_{n,m})_{n, m}$ be a doubly indexed sequence of
  random variables with values in the set of continuous functions, ie.  for all
  $n, m$,  $F_{n,m} : \Omega \longrightarrow C^0(\RR^d)$. Moreover, we assume that there
  exists a sequence of deterministic functions $(f_m)_m$ s.t. for all $n$ 
  $\EE[F_{n,m}] = f_m$ for all $m$.  We define $\overline F_{n,m} = \inv{n}
  \sum_{i=1}^n F_{i,m}$. Assume that the two following assumptions are
  satisfied
  \begin{hypo}
    \label{slln2-fm}
    One of the following criteria holds
    \begin{subhypo}
      \item \label{f_m-pointwise} The sequence $(f_m)_m$ converges pointwise to
        some continuous function $f$.
      \item \label{f_m-unif} The sequence $(f_m)_m$ converges locally uniformly to
        some function $f$.
    \end{subhypo}
  \end{hypo}
  \begin{hypo}
    \label{slln2-u}
    For any compact set $W \subset \RR^d$,
    \begin{subhypo}
      \item \label{varbar-u} $\sup_n  \sup_m  n \Var\left(\sup_{x \in W}
          \abs{\overline F_{n,m}(x)} \right) < +\infty$.
      \item \label{var-u} $\sup_n \sup_m \Var\left( \sup_{x \in W} \abs{F_{n,m}(x)}
        \right) < +\infty$.
    \end{subhypo}
  \end{hypo}
  \begin{hypo}
    \label{slln2-sup-u} 
    For all $y \in \RR^d$, $\lim_{\delta \to 0} \sup_n \sup_m \EE\left[ \sup_{\abs{x
          -y} \le \delta} \abs{F_{n,m}(x) - F_{n,m}(y)}
    \right] = 0$.
  \end{hypo}
  Then, for all functions $\rho: \NN \to \NN$, the sequence of random functions
  $\overline F_{n, \rho(n)}$ converges a.s. locally uniformly to the locally
  continuous function $f$.
\end{proposition}
\begin{remark}\label{rem:iid}
  \begin{itemize}
    \item When for every fixed $m$, the sequence $(F_{n,m})_n$ is independent
      and identically distributed, Assumption \ref{slln2-sup-u} is ensured by
      \[ \forall \, y \in \RR^d, \;\lim_{\delta \to 0} \limsup_m \EE\left[
          \sup_{\abs{x -y} \le \delta} \abs{F_{1,m}(x) - F_{1,m}(y)} \right] = 0
        \] and Assumption~\ref{var-u} implies
        \ref{varbar-u}.
    \item As in Corollary~\ref{cor:slln}, for any strictly increasing function
      $\xi: \NN \to \NN$, the sequence $\overline
      F_{\xi(n),n}$ converges a.s. locally uniformly to the locally
      continuous function $f$.
  \end{itemize}
\end{remark}
\begin{proof}
  We can apply Proposition~\ref{prop:slln2}, to deduce that a.s. $\overline
  F_{n, \rho(n)}$ converges pointwise to the function $f$. If we do not already
  know that $f$ is continuous, then thanks to \ref{var-u}, we can
  apply Lebesgue's theorem to deduce that the functions $f_m$ are
  continuous. The uniform convergence of the sequence $f_m$ to $f$ (see
  \ref{f_m-unif}) proves that the function $f$ is continuous.

  Let $W$ be a compact set of $\RR^d$, we can cover $W$ with a finite number $K$ of
  open balls $W_k$ with centers $(x_k)_k$ and radiuses $(r_k)_k$, i.e. $W_k =
  B(x_k, r_k)$ and $W = \cup_{k=1}^K W_k$. We want to prove that
  \begin{align*}
    \sup_{x \in W} \abs{\overline F_{n,\rho(n)}(x) - f(x)} 
    \xrightarrow[n \to +\infty]{a.s.} 0.
  \end{align*}
  We write
  \begin{align}
    \label{eq:split1}
    \sup_{x \in W} \abs{\overline F_{n,\rho(n)}(x) - f(x)} = \sum_{k=1}^K
    \sup_{x \in W_k} \abs{\overline F_{n,\rho(n)}(x) - f(x)}.
  \end{align}
  We split each term
  \begin{align}
    \label{eq:split-sllnu}
    \sup_{x \in W_k} \abs{\overline F_{n,\rho(n)}(x) - f(x)}= & \sup_{x \in W_k}
    \abs{\overline F_{n,\rho(n)}(x) - \overline F_{n,\rho(n)}(x_k)} +  \sup_{x \in
      W_k} \abs{f(x) - f(x_k)} \nonumber\\
    & \quad +  \abs{\overline F_{n,\rho(n)}(x_k) - f(x_k)}
  \end{align}
  Let $\varepsilon >0$. The idea is to choose the radiuses $r_k$ small enough to
  ensure that each term is controlled by a function of $\varepsilon$. Now, we
  make the idea precise. For all $k=1,\dots,K$, the last term term can be made
  smaller that $\varepsilon/K$ for $n$ larger that some $N_k$ using the
  pointwise convergence. For all $n \ge \max_{k \le K} N_k$, and all $1 \le k
  \le K$, $\abs{\overline F_{n,\rho(n)}(x_k) - f(x_k)} \le \varepsilon/K$.
  The function $f$ being continuous, it is uniformly continuous on every $W_k$.
  If we choose the $W_k$ such that their radiuses are small enough (we may need
  to increase $K$), we can ensure that for all $1 \le k \le K$
  $\sup_{x \in W_k} \abs{f(x) - f(x_k)}  \le \varepsilon/K$.
  The first term on the r.h.s of \eqref{eq:split-sllnu} deserves more attention
  \begin{align}
    \label{eq:sup-barF}
   \sup_{x \in W_k} \abs{\overline F_{n,\rho(n)}(x) - \overline
     F_{n,\rho(n)}(x_k)}  \le    \inv{n} \sum_{i=1}^n \sup_{x \in W_k}
   \abs{F_{i,\rho(n)}(x) - F_{i,\rho(n)}(x_k)}.
  \end{align}
  Now, for every $1 \le k \le K$, we want to apply Proposition~\ref{prop:slln2}
  to the sequence of random variables $\left(\sup_{x \in W_k} \abs{F_{n,m}(x) -
      F_{n,m}(x_k)}\right)_{n,m}$. Assumption~\ref{slln2} is clearly
  satisfied using Minkowski's inequality.

  Let us define the sequence $(Y_{n,m})_{n,m}$ by 
  $$
  Y_{n,m} = \sup_{x \in W_k} \abs{F_{n,m}(x) - F_{n,m}(x_k)} -
  \EE\left[ \sup_{x \in W_k} \abs{F_{n,m}(x) - F_{n,m}(x_k)} \right],
  $$
  satisfying $\EE[Y_{n,m}] = 0$ and the assumptions of
  Proposition~\ref{prop:slln2}.
    Hence, it yields that
  \begin{align}
    \label{eq:limempirique}
    \lim_{n \to +\infty}  \inv{n} \sum_{i=1}^n \sup_{x \in W_k}
    \abs{F_{i,\rho(n)}(x) - F_{i,\rho(n)}(x_k)} - \EE\left[ \sup_{x \in W_k}
      \abs{F_{n,\rho(n)}(x) - F_{n,\rho(n)}(x_k)} \right] = 0.
  \end{align}
  From \ref{slln2-sup-u}, we know that if the $W_k$ are chosen small
  enough, 
  \begin{equation}
    \label{eq:limsup}
    \sup_n \EE\left[ \sup_{x \in W_k} \abs{F_{n,\rho(n)}(x) - F_{n,\rho(n)}(x_k)} 
    \right] \le \varepsilon / K.
  \end{equation}
  Then, combining \eqref{eq:sup-barF}, \eqref{eq:limempirique}
  and \eqref{eq:limsup} yields that
  $ \sup_{x \in W_k} \abs{\overline F_{n,\rho(n)}(x) - \overline F_{n,\rho(n)}(x_k)}  \le
  \varepsilon / K$.
  We plus this inequality into \eqref{eq:split-sllnu} and deduce
  from~\eqref{eq:split1}, that for $n$ large enough,
  \begin{equation*}
    \sup_{x \in W} \abs{\bar F_{n, \rho(n)}(x) - f(x)} \le 3 \varepsilon. \qedhere
  \end{equation*}
\end{proof}


\section{Numerical experiments}\label{sec:num}

\subsection{Practical implementation}
\label{sec:implem}

Our approach cleverly mixes the famous multilevel Monte Carlo technique with importance
sampling to reduce the variance. A classical approach would have been to consider the
multilevel approximation of  $\EE \left[ \psi(X_T(\t)) \expp{-\t \cdot W_T - \frac{1}{2}
    |\t|^2 T}\right]$ while choosing the value of $\t$ which minimizes the variance of the
central limit theorem for multilevel Monte Carlo (see~\cite{BAK2015}). The asymptotic
variance involves both $\nabla \psi$ and the process $U$ given in \eqref{eq_U_theta}. Hence, a classical approach to
importance sampling for multilevel Monte Carlo would require extra knowledge than the
function $\psi$ and the underlying process $X$, thus precluding any kind of
automation.

We have chosen a completely different approach allowing for one importance sampling
parameter per level, which enables us to treat each level independently of the others.
In each level, we use a sample average approximation as in~\cite{JourLel} to compute the
optimal importance sampling parameter defined as the one minimizing the variance of the
current level. From Theorem~\ref{thm:clt-ml}, we know that this approach is optimal in the
sense that our multilevel estimator $Q_L(\hat \lambda_0, \dots, \hat \lambda_L)$ satisfies
a central limit theorem with a limiting variance given by $\inf v$ where $v$ defined
by~\eqref{eq:variance-ml} is the variance of the standard multilevel Monte Carlo
estimator. We managed to provide an algorithm reaching the optimal limiting variance
without computing $\nabla \psi$ nor the process $U$, hence our approach can be made fully
automatic.

\paragraph{Computation of $\hat \la_\l$.}

The parameters $\hat \la_\l$ are defined as the solutions of strongly convex minimization
problems. The minimization step is performed by the Newton--Raphson algorithm to $\nabla
v_{\l, N_\l'}$. The samples required to compute $\nabla v_{\l, N_\l'}$ and $\nabla^2
v_{\l, N_\l'}$ are generated once and for all before starting the Newton--Raphson
procedure such that the same samples are used through all the iterations of the gradient
descent.  This feature is specific to the optimisation step and may make the algorithm
highly memory demanding as soon as the numbers $N'_\l$ become large.  As the parameter
$\lambda$ is not involved in the function $\psi$, all the quantities $\psi(
X^{m^{\ell}}_{T,\l,k}) - \psi( X^{m^{\ell-1}}_{T,\l,k})$ for $k=1,\dots,N_\l$ can be
precomputed before starting the minimization algorithm, which enables us to save a lot of
computational time. 

The efficiency of the Newton--Raphson algorithm very much depends on the convexity of the
$v_{\l, N_\l'}$ functions. As already pointed out in~\cite{JourLel}, the smallest
eigenvalue of the Hessian matrix $\nabla^2  v_{\l, N_\l'}$ is basically
$\frac{T}{N_\l'} \sum_{k=1}^{N_\l'} \frac{m^\l}{(m-1)T} \abs{\psi( X^{m^{\ell}}_{T,\l, k}) -
\psi( X^{m^{\ell-1}}_{T,\l, k})}^2 \Ec^+(W_{\l, k}, \la)$, which can become extremely small and
then conflicts with the will to have the strongest possible convexity in order to speed up
Newton--Raphson's algorithm. This difficulty is circumvented by noticing the equality
$\nabla v_{\l, N_\l'}(\hat \la_\l) = 0$ can be written as 
\begin{align*}
  \hat \la_\l T - \frac{\frac{1}{N_\l'} \sum_{k=1}^{N_\l'}
    \frac{m^\l}{(m-1)T} W_{k, \l, T} \abs{\psi( X^{m^{\ell}}_{T,\l,k}) - \psi(
    X^{m^{\ell-1}}_{T,\l,k})}^2
  \expp{-\hat \lambda_\l \cdot W_{T,\l,k}}}{\frac{1}{N_\l'} \sum_{k=1}^{N_\l'}
    \frac{m^\l}{(m-1)T} \abs{\psi( X^{m^{\ell}}_{T,\l,k}) - \psi( X^{m^{\ell-1}}_{T,\l,k})}^2
  \expp{-\hat \lambda_\l \cdot W_{T,\l,k}}} = 0.
\end{align*}
Hence, $\hat \la_\l$ can be interpreted as the root of $\nabla u_{\l, N_\l'}$ with
\[
  u_{\l, N_\l'}(\lambda) = \frac{\abs{\lambda}^2 T}{2} 
  + \log\left( \frac{1}{N_\l'} \sum_{k=1}^{N_\l'}
    \frac{m^\l}{(m-1)T} \abs{\psi( X^{m^{\ell}}_{T,\l,k}) - \psi( X^{m^{\ell-1}}_{T,\l,k})}^2
    \expp{-\lambda \cdot W_{T,\l,k}} \right).
\]
The Hessian matrix of $u_{\l, N_\l'}$ is given by
\begin{align}
  \label{eq:hes_ml}
  \nabla^2 u_{\l, N_\l'}(\lambda) = & T I_q + \frac{\frac{1}{N_\l'} \sum_{k=1}^{N_\l'}
  \frac{m^\l}{(m-1)T} W_{k, \l, T} {(W_{k, \l, T})}^* \abs{\psi( X^{m^{\ell}}_{T,\l,k}) - \psi( X^{m^{\ell-1}}_{T,\l,k})}^2
    \expp{-\lambda \cdot W_{T,\l,k}}}{\frac{1}{N_\l'} \sum_{k=1}^{N_\l'}
    \frac{m^\l}{(m-1)T} \abs{\psi( X^{m^{\ell}}_{T,\l,k}) - \psi( X^{m^{\ell-1}}_{T,\l,k})}^2
    \expp{-\lambda \cdot W_{T,\l,k}}} \nonumber \\
  -  & \frac{
    \left(\frac{1}{N_\l'} \sum_{k=1}^{N_\l'}
      \frac{m^\l}{(m-1)T} W_{k, \l, T} \abs{\psi( X^{m^{\ell}}_{T,\l,k}) - \psi( X^{m^{\ell-1}}_{T,\l,k})}^2
      \expp{-\lambda \cdot W_{T,\l,k}}\right)}
  {\frac{1}{N_\l'} \sum_{k=1}^{N_\l'}
    \frac{m^\l}{(m-1)T} \abs{\psi( X^{m^{\ell}}_{T,\l,k}) - \psi( X^{m^{\ell-1}}_{T,\l,k})}^2
    \expp{-\lambda \cdot W_{T,\l,k}}} \nonumber \\
  & \quad \frac{\left(\frac{1}{N_\l'} \sum_{k=1}^{N_\l'}
      \frac{m^\l}{(m-1)T} W_{k, \l, T} \abs{\psi( X^{m^{\ell}}_{T,\l,k}) - \psi( X^{m^{\ell-1}}_{T,\l,k})}^2
      \expp{-\lambda \cdot W_{T,\l,k}}\right)^*}
  {\frac{1}{N_\l'} \sum_{k=1}^{N_\l'}
    \frac{m^\l}{(m-1)T} \abs{\psi( X^{m^{\ell}}_{T,\l,k}) - \psi( X^{m^{\ell-1}}_{T,\l,k})}^2
    \expp{-\lambda \cdot W_{T,\l,k}}}.
\end{align}
From the Cauchy Schwartz inequality, it is clear that $\nabla^2 u_{\l,
  N_\l'}(\lambda)$ is lower bounded by $T I_q$, where the inequality is to be
understood in the sense of the order on symmetric matrices.

\paragraph{Description of the algorithm.} Our algorithm splits in two steps: the
minimization step to compute the optimal importance sampling measure and the MLMC
step to actually provide an estimator of $\EE[\psi(X_T)]$. The samples used in the two
steps are independent.  For the sake of clearness, we provide the pseudocode of
our global method in in Algorithm~\ref{algo:MLIS}.
\begin{algorithm}[ht]
  Generate $X^{m^0}_{T, 0, 1}, \dots, X^{m^0}_{T, 0, N_0'}$ i.i.d. samples
  following the law of $X^{m^0}_T$ independently of the other blocks. \;
  Solve $\nabla u_{0, N_0'}(\hat \lambda_0) = 0$ by using the Newton--Raphson
  algorithm. \;
  \For{$\l=1:L$}{
    Generate $(X^{m^\l}_{T, \l, 1}, X^{m^{\l-1}}_{T, \l, 1}), \dots,
    (X^{m^\l}_{T, \l, N_\l'}, X^{m^{\l-1}}_{T, \l, N_\l'})$ i.i.d. samples following the law
    of $(X^{m^\l}_T, X^{m^{\l-1}}_T)$ independently of the other blocks. \;
    Solve $\nabla u_{\l, N_\l'}(\hat \lambda_\l) = 0$ by using the Newton--Raphson
    algorithm.
  }
  Conditionally on $\hat \la_0$, generate $\tilde X^{m^0}_{T, 0, 1}(\hat
  \la_0), \dots, \tilde X^{m^0}_{T, 0, N_0}(\hat \la_0)$ i.i.d. samples with the
  law of $X^{m^0}_T(\hat \la_0)$ independently of the other blocks. The tilde and non
  tilde quantities are conditionally independent. \;
  \For{$\l=1:L$}{
    Conditionally on $\hat \la_\l$, generate $(\tilde X^{m^\l}_{T, \l, 1}(\hat
    \la_\l), \tilde X^{m^{\l-1}}_{T, \l, 1}(\hat \la_\l)), \dots, (\tilde X^{m^\l}_{T, \l, N_\l}(\hat
    \la_\l), \tilde X^{m^{\l-1}}_{T, \l, N_\l}(\hat \la_\l))$ i.i.d. samples with the law of
    $(X^{m^\l}_T(\hat \la_\l), X^{m^{\l-1}}_T(\hat \la_\l))$ independently of the other
    blocks. The tilde and non tilde quantities are conditionally independent. \;
  }
  Compute the multilevel importance sampling estimator
  \begin{align*}
    Q_L(\hat \la_0,\dots,\hat \la_L) & =\frac{1}{N_0}\sum_{k=1}^{N_0}
    \psi( \tX^{m^0}_{T,0,k}(\hat \la_0)) \Ec^-(\tilde W_{0,k}, \hat \la_0) \\
    & \qquad +\sum_{\ell=1}^{L}\frac{1}{N_\ell}
    \sum_{k=1}^{N_\ell}\left(\psi( \tX^{m^{\ell}}_{T,\ell,k}(\hat \la_{\ell}))-\psi(
    \tX^{m^{\ell-1}}_{T,\ell,k}(\hat \la_{\ell}))\right) \Ec^-(\Tilde W_{\ell,k},
    \hat \la_\ell).
  \end{align*} 
  \caption{Multilevel Importance Sampling (MLIS)}
  \label{algo:MLIS}
\end{algorithm}

\paragraph{Complexity analysis.} In this paragraph, we focus on the impact of the number
of levels $L$ on the overall computational time of our algorithm. The computational cost
of the standard multilevel estimator is proportional to
\begin{equation*}
  C_{ML} = \sum_{\l=0}^L N_\l m^\l = m^{2 L +1} L^2.
\end{equation*}
The global cost of our algorithm writes as the sum of the cost of the computation of
the $(\hat \lambda_{\l})_\l$ and of the standard multilevel estimator
\begin{equation*}
  C_{MLIS} = \sum_{\l=0}^L N_\l' (m^\l + 3 K_\l) + \sum_{\l=0}^L N_\l m^\l
\end{equation*}
where $K_\l$ is the number of iterations of Newton--Raphson's algorithm to approximate
$\hat \lambda_\l$ and the factor $3$ corresponds to the fact that building $\nabla
u_{\l, N_\l'}$ and $\nabla^2 u_{\l, N_\l'}$ basically boils down to three Monte Carlo
summations. In practice, $K_\l \le 5$ as the problem is strongly convex. Because the same
random variables are used at each iteration of the optimisation step, they must be stored,
which makes the memory footprint of our algorithm proportional to $N'_\l$.

So, if we choose $N_\l' = \frac{N_\l m^\l}{m^\l + 15}$, the total cost of our MLIS
algorithm should be roughly twice the cost of the standard multilevel estimator. This
choice of $N_\l'$ reduces the number of samples used to approximate the variance of the
first levels compared to using directly $N_\l$. However, when $L$ increases, $N'_\l$ can
become extremely large for small values of $\l$ which leads to an even larger memory
footprint (see Section~\ref{sec:implem}). Not to break the scalability of the algorithm,
the values of $N'_\l$ have to be kept reasonable depending on the amount of memory
available on the computer. For an instance, enforcing $N'_\l \le 500000$ is reasonable on
a computer with $8Gb$ of RAM. Anyway, it is crystal clear that a fairly good approximation
of the variance $v_\l$ is enough and running for an ultimately accurate estimator would
lead to a tremendous waste of computational time. Monitoring the convergence of
$v_{\l, N_\l'}$ would really help choosing sensible values for $N_\l'$.

\subsection{Comparison with existing algorithms}

In Theorem~\ref{thm:clt-ml},  we obtain the same limiting variance as in~\cite{BHK13-1},
in which the authors apply MLMC to importance sampling (see~\eqref{eq:is-ml}) and not
vice--versa as we do. The way importance sampling and MLMC are coupled does not actually
matter in terms of convergence rate but it does matter in practice. First, our approach
preserves the independence of the different levels by solving one optimization problem
per level instead of a global one. Hence, the contributions of the different levels are
computed independently of each other as in the standard MLMC setting. Second, we use
sample average approximation combined with Newton--Raphson's algorithm to compute the best
importance parameters, whereas in \cite{BHK13-1,BHK16}, the authors rely on stochastic
approximation, which is known to demand proper tuning to effectively converge in practice.
Our approach inherits from the good convergence properties of Newton's algorithm when
applied to strongly convex problems with a tractable Hessian matrix. As already noted in
\cite{JourLel}, this approach is more stable and robust.

\subsection{Experimental settings}

We compare four methods in terms of their root mean squared error (RMSE): the crude Monte
Carlo method (MC), the adaptive Monte Carlo method proposed in~\cite{JourLel} (MC+IS), the
Multilevel Monte Carlo method (ML) and our Importance Sampling Multilevel Monte Carlo
estimator (ML+IS). We recall that the RMSE is defined by $RMSE = \sqrt{\text{Bias}^2 +
  \text{Variance}}$. In the computation of the bias, the true value is replaced by its
multilevel Monte Carlo estimator with $L=9$ levels, which yields a very accurate
approximation. Not to mention, the CPU times showed on the graphs take into account both
the time to the search for the optimal parameter and the time for the second stage Monte
Carlo, be it multilevel or not.

\subsection{Multidimensional Dupire's framework}

We consider a $d-$dimensional local volatility model, in which the dynamics,
under the risk neutral measure, of each asset $S^i$ is supposed to be given by
\begin{equation*}
  dS_t^i = S_t^i (r \dt + \sigma(t, S_t^i) dW^i_t), \qquad S_0 = (S_0^1, \dots, S_0^d)
\end{equation*}
where $W = (W^1, \dots, W^d)$, each component $W^i$ being a standard Brownian
motion with values in $\RR$. For the numerical experiments, the covariance structure of $W$ will be
assumed to be given by $\langle W^i, W^j \rangle_t = \rho t \ind{i \neq j} + t
\ind{i = j}$. We suppose that $\rho \in (-\frac{1}{d-1}, 1)$, which
ensures that the matrix $C=(\rho\ind{i \neq j} +\ind{i = j})_{1\leq i,j\leq
  d}$ is positive definite. Let $L$ denote the lower triangular matrix
involved in the Cholesky decomposition $C=LL^*$. To simulate $W$ on the
time-grid $0<t_1<t_2<\hdots<t_N$, we need $d\times N$ independent
standard normal variables and set
$$\begin{pmatrix}
  W_{t_1} \\ W_{t_2}\\ \vdots \\ W_{t_{N-1}} \\ W_{t_N}
\end{pmatrix}=
\begin{pmatrix}
  \sqrt{t_1}L & 0 & 0 &\hdots &0\\
  \sqrt{t_1}L &\sqrt{t_2-t_1}L & 0 &\hdots &0\\
  \vdots&\ddots&\ddots&\ddots&\vdots\\
  \vdots&\ddots&\ddots& \sqrt{t_{N-1}-t_{N-2}}L&   0 \\
  \sqrt{t_1}L & \sqrt{t_2-t_1}L  &\hdots & \sqrt{t_{N-1}-t_{N-2}}L
  &\sqrt{t_N-t_{N-1}}L
\end{pmatrix}G,$$
where $G$ is a normal random vector in $\RR^{d \times N}$.  The maturity time and
the interest rate are respectively denoted by $T>0$ and $r>0$.  The local
volatility function $\sigma$ we have chosen is of the form 
\begin{equation}
  \label{locvol}
  \sigma(t, x) = 0.6 (1.2-\expp{-0.1 t} \expp{-0.001 (x \expp{r t} - s)^2 })
  \expp{-0.05 \sqrt{t}},
\end{equation}
with $s>0$. We know that there exists a duality between the variables $(t ,x)$
and $(T, K)$ in Dupire's framework. Hence for formula~\eqref{locvol} to
make sense, one should choose $s$ equal to the spot price of the underlying
asset so that the bottom of the smile is located at the forward money. We
refer to Figure~\ref{fig:volsmile} to have an overview of the smile.
\begin{figure}[ht]
  \centering \includegraphics[width=8cm]{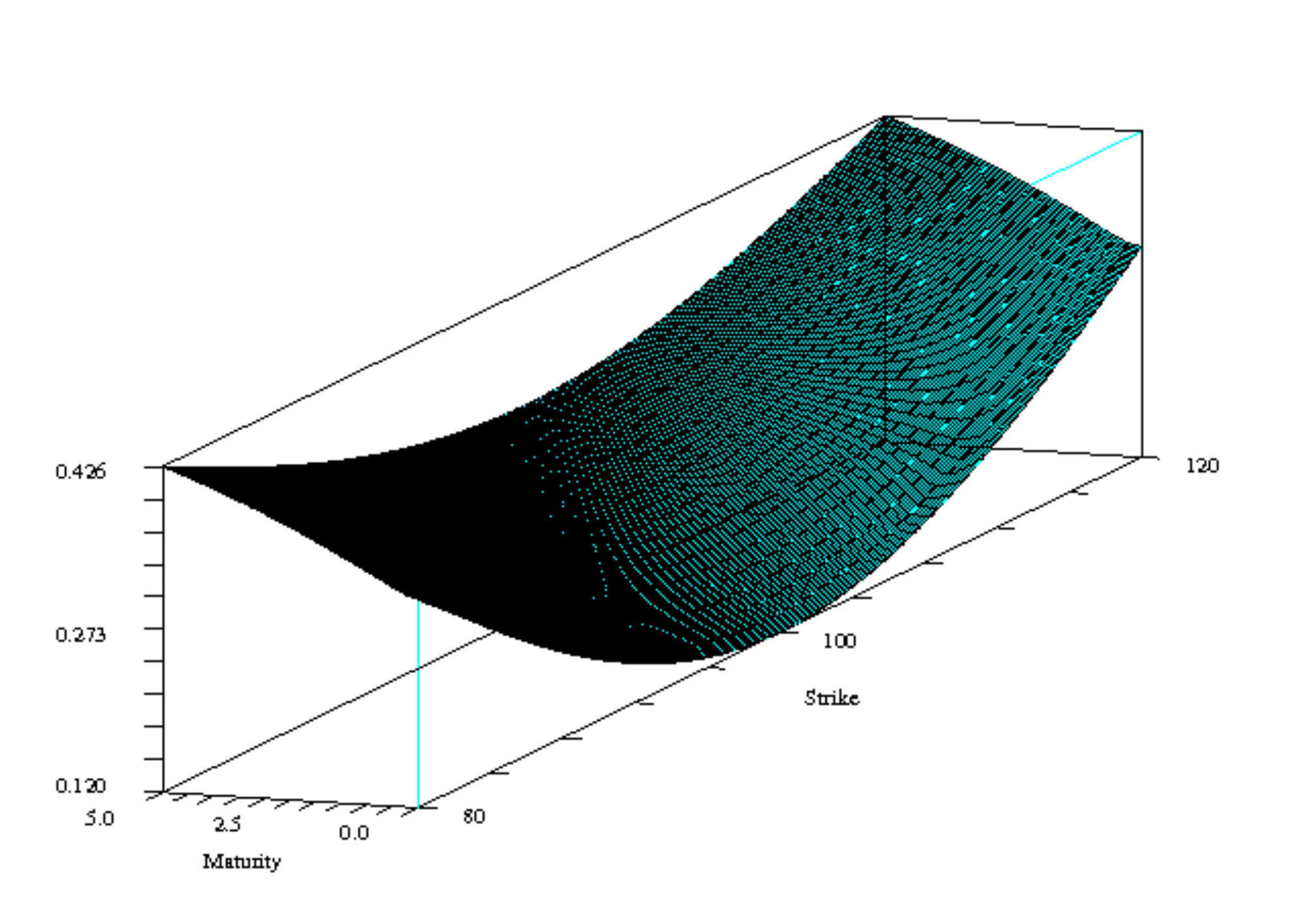}
  \caption{Local volatility function \label{fig:volsmile}}
\end{figure}

\paragraph{Basket option}

We consider options with payoffs of the form $(\sum_{i=1}^d \omega^i S_T^i - K)_+$ where
$(\omega^1, \dots, \omega^d)$ is a vector of algebraic weights. The strike value $K$ can
be taken negative to deal with Put like options. With no surprise, we can see on
Figure~\ref{fig:basket-loc} that multilevel estimators always outperform their classical
Monte Carlo counterpart. The comparison for very little accurate estimators may be
meaningless as it is pretty difficult to reliably measure short execution times and the
empirical variance of the estimator is in this case even less accurate than the estimator
itself. Note that the points on the extreme right hand side are obtained for multilevel
estimators with $L=2$, respectively for Monte Carlo estimators with $256$ samples. For
RMSE between $0.1$ and $0.005$, our MLIS estimator is $10$ times faster than the standard
ML estimator. When a very high accuracy is required, namely when RMSE is smaller than
$0.001$, the MLIS estimator remains between $3$ and $4$ times faster than the standard
multilevel estimator, which is already a great achievement since for this level of
accuracy, the ML estimator may need several dozens of minutes to yield its result.

\begin{figure}[ht]
  \centering \includegraphics[scale=0.6]{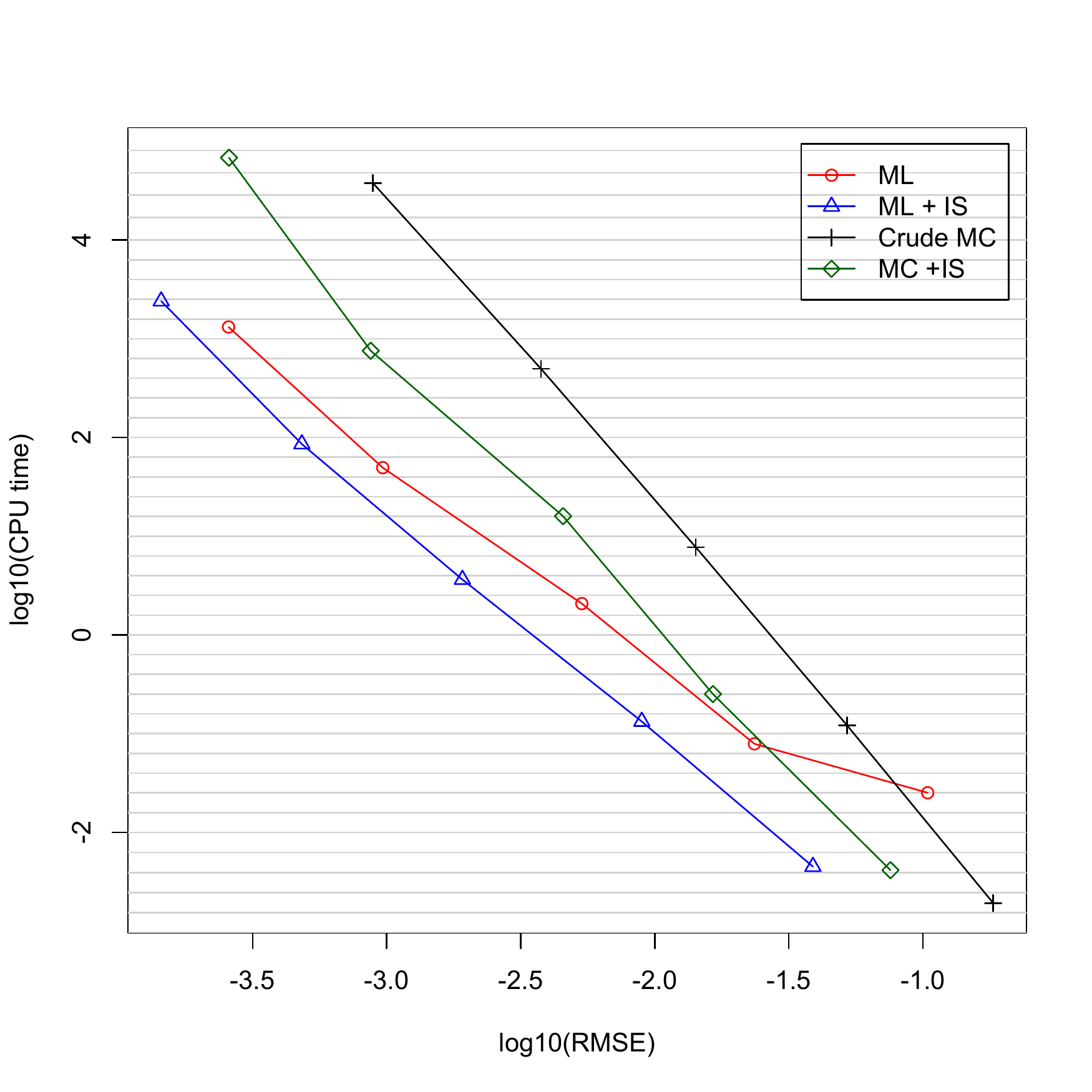}
  \caption{$\sqrt{MSE}$ vs. CPU time for a basket option in the local volatility model
    with $I=5$, $r=0.05$, $T=1$, $S_0=100$, $K=100$, $m=4$.
  \label{fig:basket-loc}}
\end{figure}

\subsection{Multidimensional Heston model}

The multidimensional Heston model can be easily written by specifying on the one
hand that each asset follows a 1-D Heston model and on the other hand the correlation
structure between the involved Brownian motions. The asset price process $S =
(S^1, \dots, S^d)$ and the volatility process $\sigma=(\sigma^1, \dots,
\sigma^d)$ solve
\begin{align*}
  dS^i_t &= r S^i_t dt + \sqrt{\sigma_t^i} S^i_t dB^i_t \\
  d\sigma^i_t &= \kappa^i (a^i - \sigma_t^i) dt + \nu^i_t \sqrt{\sigma^i_t} (\gamma^i dB^i_t + \sqrt{1 -
    (\gamma^i)^2} d\tilde B^i_t) 
\end{align*}
where all the components of $B=(B^1, \dots, B^d)$ and $\tilde B = (\tilde B^1, \dots,
\tilde B^d)$ are real valued Brownian motions. The vectors $\kappa=(\kappa^1, \dots,
\kappa^d)$ and $a=(a^1, \dots, a^d)$ denote respectively the reversion rate and the mean
level of each volatility process, while the vector $\nu$ is the volatility of the
volatility process.  The vector $\bar \gamma = (\gamma^1, \dots, \gamma^d)$ embodies the
correlations between an asset and its volatility process, with $\gamma^i \in ]-1,1[$ for
all $1 \le i \le d$. The vector valued processes $B$ and $\tilde B$ are independent and
satisfy
\begin{equation*}
  d\langle B \rangle_t = \Gamma_S \dt \quad \text{and} \quad
  d\langle \tilde B \rangle_t = I_d \dt
\end{equation*}
where we assume for our experiments that the covariance matrix
$\Gamma_S$ has the structure
\begin{equation}
  \label{eq:covstruct}
\Gamma_S = \begin{pmatrix}
  1 & \rho & \hdots & \rho\\
  \rho & 1 &\ddots & \vdots\\
  \vdots&\ddots&\ddots& \rho\\
  \rho &\hdots & \rho & 1 
\end{pmatrix}
\end{equation}
with $\rho \in \left]\frac{-1}{I-1}, 1\right[$, such that the matrix $\Gamma_S$ is positive definite.
The processes $B$ and $\tilde B$ are Wiener processes with covariance matrices
given by $\Gamma_S$ and $I_d$ respectively. 

For the sake of simplicity, we decided not to add any extra correlation between the
components of $\tilde B$, hence the choice $d\langle \tilde B \rangle = I_d \, dt$ and we
assume in the following that for all the $\gamma^i$'s are equal for $1 \le i \le d$,
$\gamma^i = \gamma$. The correlations between the volatilities are entirely
specified by the correlations between the assets. Even though we do not aim at discussing
the correlation structure of the multidimensional Heston model, we believe it is important
to make precise the underlying correlation structure in the multidimensional model so that
the experiments are easily reproducible.

The model can be equivalently written
\begin{align*}
  dS^i_t &= r S^i_t dt + \sqrt{\sigma_t^i} S^i_t dB^i_t \\
  d\sigma^i_t &= \kappa^i (a^i - \sigma_t^i) dt + \nu^i_t \sqrt{\sigma^i_t} dW^i_t 
\end{align*}
where the processes $W$ and $B$ are Wiener processes satisfying
\begin{align*}
  d\langle B \rangle_t = \Gamma_S \dt; \;
  d\langle B, W \rangle_t = \gamma \Gamma_S \dt; \; 
  d\langle W \rangle_t = (\gamma^2 \Gamma_S + (1 - \gamma^2) I_d) \dt.
\end{align*}
The process $(B, W)$ with values in $\RR^{2 d}$ is a Wiener process with
covariance matrix
$$
\Gamma = \begin{pmatrix}
  \Gamma_S & \gamma \Gamma_S \\
  \gamma \Gamma_S & \gamma^2 \Gamma_S + (1 - \gamma^2) I_d
\end{pmatrix}.
$$
Hence, the pair of processes $(B, W)$ can be easily simulated by applying the
Cholesky factorization of $\Gamma$ to a standard Brownian motion with values in
$\RR^{2 d}$.

\paragraph{Basket Option} We consider a basket option as in the local volatility model.
Figure~\ref{fig:basket_heston} looks very much the same as in the case of the local
volatility model (see Figure~\ref{fig:basket-loc}). The MLIS estimator always outperforms
all the ML estimator by a factor of $3$ to $4$. Note that for small RMSE, the computational
time can go beyond several hours, hence cutting it down by two or three times represents a
real improvement.

\begin{figure}[ht]
  \centering
  \includegraphics[scale=0.6]{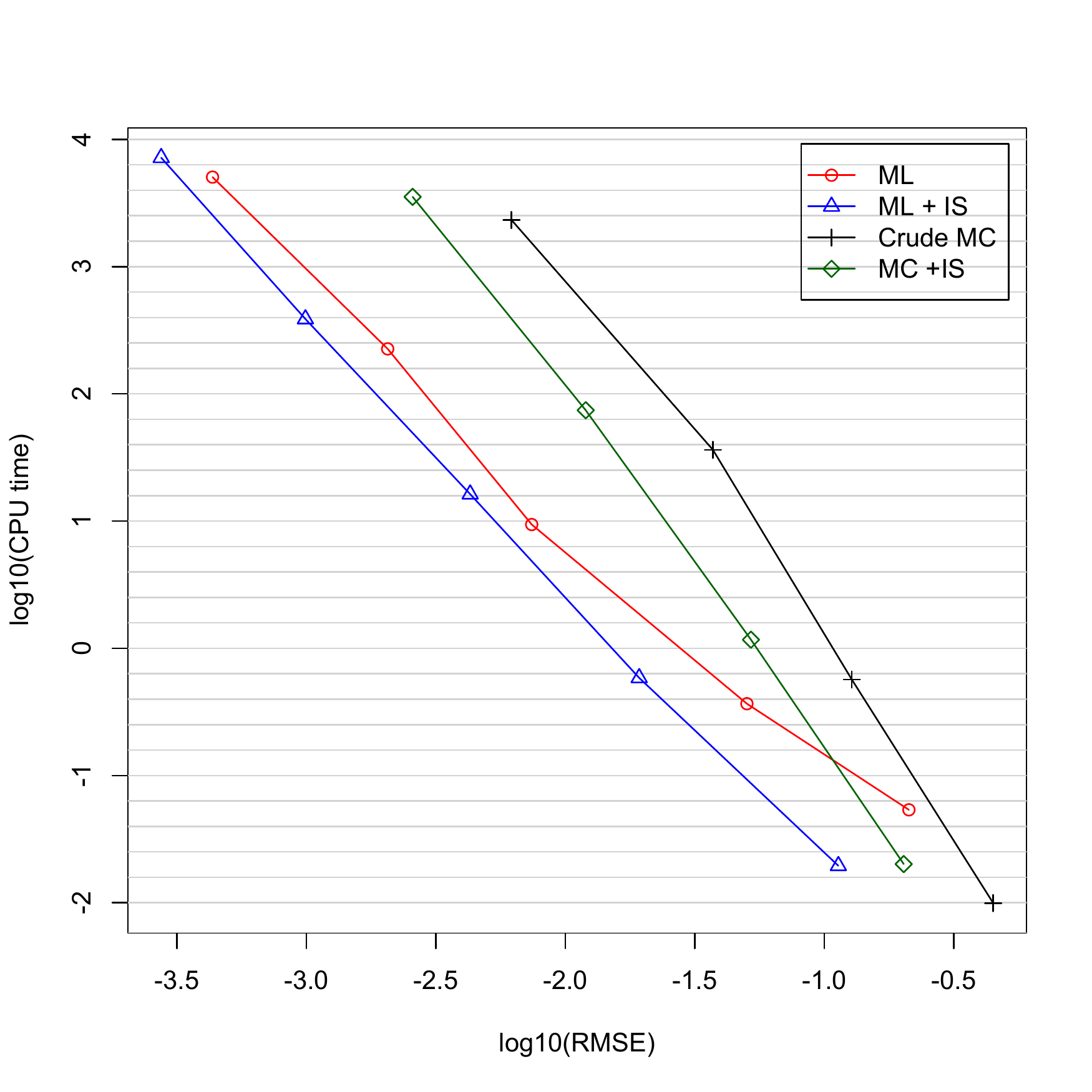}
  \caption{$\sqrt{MSE}$ vs. CPU time for a best of option in the multidimensional Heston
    model with $I=10$, $r=0.03$, $T=1$, $S_0=100$, $K=100$, $\nu=0.01$,
    $\kappa=2$, $a = 0.04$, $\gamma=-0.2$, $\rho=0.3$ and $m=4$.}
  \label{fig:basket_heston}
\end{figure}

\paragraph{Best of option} We consider options with payoffs of the form $(\max_{1 \le i \le
  d} S_T^i - K)_+$. The payoff of this option does obviously not satisfy the assumptions
of Theorem~\ref{thm:slln-ml} as the payoff of the ``best of'' options is not Hölder with
$\alpha \ge 1$. Nonetheless, the multilevel approach beats the standard Monte Carlo
technology by far (see Figure~\ref{fig:bestof}). Moreover, coupling importance sampling
with the multilevel approach improves the accuracy. For a fixed RMSE, we can expect MLIS
to be $3$ faster that ML. This example shows the robustness of the method, which performs
well whereas the theoretical assumptions are not satisfied.

\begin{figure}[ht]
  \centering \includegraphics[scale=0.6]{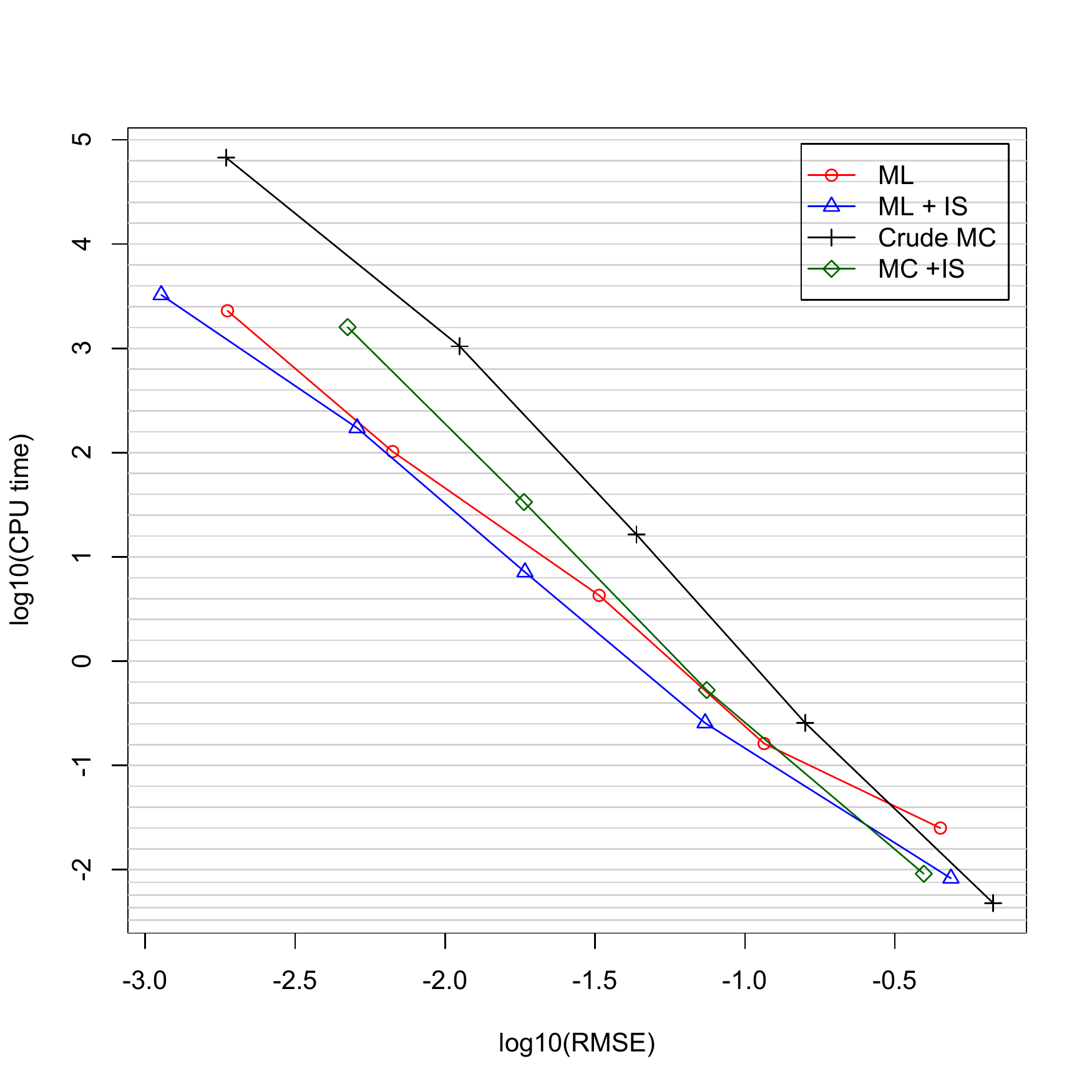}
  \caption{$\sqrt{MSE}$ vs. CPU time for a best of option in the multidimensional Heston
    model with $I=5$, $r=0.03$, $T=1$, $S_0=100$, $K=140$, $\nu=0.25$,
    $\kappa=2$, $a = 0.04$, $\gamma=0.2$, $\rho=0.5$ and $m=4$.}
  \label{fig:bestof}
\end{figure}

\section{Conclusion}

We have presented a new estimator making the most of the recent works on multilevel Monte
Carlo and on adaptive importance sampling. As expected, this new estimator outperforms the
standard multilevel Monte Carlo estimator by a great deal. For a fixed accuracy measured
in terms the mean squared error, the MLIS estimator is between $3$ and $10$ times faster
that the standard multilevel Monte Carlo estimator. This efficiency of our MLIS approach
could still be improved by monitoring the number of samples $N_\l'$ to be used to
approximate the variance $v_{\l, N_\l'}$ in each level. Actually, we believe that there is
no need to compute a too accurate approximation of this variance as a slight decrease in
the accuracy of $\hat \lambda_\l$ would not lead to a serious deterioration of the accuracy
of the MLIS estimator but it could help save a lot of computational time.

\section*{Acknowledgment}

We are grateful to the anonymous referees for their valuable comments and suggestions,
which helped us greatly improve the paper.

\clearpage
\appendix

\section{Auxiliary lemmas}

\subsection{Central limit theorems for martingale arrays}

\begin{theorem}[Central limit theorem for triangular array]
  \label{thm:lindeberg}
  Suppose that $(\Omega,\mathbb F, \PP)$ is a probability space and that for
  each $n$, we have a filtration $\mathbb{F}_{n}=(\mathcal{F}_{k}^{n})_{k\geq
  0}$, a sequence $k_{n} \longrightarrow \infty  \mbox{ as  } n \longrightarrow
  \infty$  and a real  vector martingale
  $Y^{n}=(Y_{k}^{n})_{k\geq 0}$ adapted to $\mathbb{F}_{n}$.
  We make the following two assumptions.
  \begin{hypo}
    \label{triarray}
    \begin{subhypo}
    \item \label{bracket} There exists a deterministic symmetric positive
      semi-definite matrix $\varGamma$, such that
      $$
      \langle Y^{n} \rangle_{k_{n}}= \sum_{k=1}^{k_{n}} \mathbb{E}\left[
      |Y_{k}^{n}-Y_{k-1}^{n}|^{2} | \mathcal{F}_{k-1}^{n}  \right]
      \overset{\mathbb{P}}{\underset{n\rightarrow\infty}{\longrightarrow}}  \varGamma.
      $$
    \item \label{ui} There exists a real number $a>1$, such that
      $$ 
      \sum_{k=1}^{k_{n}} \mathbb{E} \left[| Y_{k}^{n}-Y_{k-1}^{n}|^{2a} |
      \mathcal{F}_{k-1}^{n} \right]
      \overset{\mathbb{P}}{\underset{n\rightarrow\infty}{\longrightarrow}}  0.
      $$
    \end{subhypo}
  \end{hypo}

  Then 
  $$
  Y_{k_{n}}^{n} \xrightarrow{\mathcal{L}}
  \mathcal{N}(0,\varGamma)\quad\mbox{ as }\; n\rightarrow\infty.
  $$ 
\end{theorem}

\subsection{Asymptotic behavior of the process $\left(X^{m^{\l}}-X^{m^{\l -1}}\right)_{\l\geq 0}$}
In the following we recall some results around the stable convergence.
 Let $Z_n$ be a sequence of random variables with values in a Polish space $E$, all defined on the same probability 
space $(\Omega,\mathcal F,\PP)$. Let $(\tilde \Omega,\tilde {\mathcal F},\tilde \PP)$ be an extension 
of $(\Omega,\mathcal F,\PP)$, and let $Z$ be an $E$-valued random variable on the extension. 
We say that  $(Z_n)$ converges in law to $Z$ stably and write $Z_n\Longrightarrow^{stably}Z$, if
\begin{equation*}
\displaystyle \mathbb E(Uh(Z_n))\rightarrow \tilde {\mathbb E}(Uh(Z))
\end{equation*}
for all $h:E\rightarrow \mathbb R$ bounded continuous and all bounded random variable $U$
on $(\Omega,\mathcal F)$ .  According to Section 2 of Jacod
\cite{Jacod} and Lemma 2.1 of Jacod and Protter \cite{Jacod-Protter}, we have the
following result
\begin{lemma}\label{lemma}
Let  $V_n$ and $V$ be defined on $(\Omega,\mathcal F)$ with values
 in another metric space.
\begin{equation*}
\mbox{If }\;\; V_n\overset{\mathbb P}{\rightarrow}V,\;\;Z_n\Longrightarrow^{stably}Z\;\;\;\mbox{then}\;\;\; (V_n,Z_n)\Longrightarrow^{stably}(V,X).
\end{equation*} 
\end{lemma}
The following result proved by Ben Alaya and Kebaier \cite[Theorem 3]{BAK2015} is an improvement of Theorem 3.2 of Jacod and Protter \cite{Jacod-Protter}, for the setting of Multilevel Euler scheme. 
More precisely, if $(X^{m^{\l}}_t)_{t\geq 0}$ denotes the Euler scheme with time step
$m^{\l}$, with  $m, \l\in\NN\setminus\{0,1\}$, then we have the following weak convergence in the Skorohod topology. 
\begin{theorem}\label{th-acc}
 Assume that $b$ and $\sigma$ are $\mathcal C^1$ with linear growth then the following result holds. 
\begin{equation*}
\text{For all}\;m\in\mathbb N\setminus\{0,1\},\quad\sqrt{\frac{m^\l}{(m-1)T}}(X^{m^\l}-X^{m^{\l-1}})\Longrightarrow^{stably} U,\quad \text{ as }\l\rightarrow\infty,
\end{equation*}
with $(U_t)_{0\leq t\leq T}$ the $d$-dimensional diffusion process solution to (\ref{eq_U_theta})
\end{theorem}

\begin{proposition}
  \label{prop:conv-moments-level}
  Let $\psi : \RR^d \to \RR$ be a $C^1$ function such that $\psi \in \Hc_\alpha$,
  for some $\alpha \ge 1$ and $\nabla \psi$ has at most polynomial growth. 
  For any real valued random variable $Y$ defined on $(\Omega,
  \Fc)$ such that $\EE[\abs{Y}^{1 + \eta}]$, for some $\eta>0$, we have, for any
  $\delta>0$
  \[
    \EE\left[ \left(\frac{m^\l}{(m-1)T}\right)^{\delta/2}
      \left(\psi(X^{m^{\ell}}_{T})-\psi(X^{m^{\ell-1}}_{T})\right)^\delta
      Y \right] \xrightarrow[\l \to +\infty]{} 
    \EE\left[ \left(\nabla \psi(X_T) \cdot U_T\right)^\delta Y \right].
  \]
\end{proposition}
\begin{proof}
  The Taylor expansion applied to the real valued function $\psi$ yields
  \begin{align*}
    \psi(X^{m^{\ell}}_{T})-\psi( X^{m^{\ell-1}}_{T}) = &
    \nabla \psi(X_T)\cdot(X^{m^{\ell}}_{T}-X^{m^{\ell-1}}_{T}) \\
    & + (X^{m^{\ell}}_{T}-X_T) \cdot \varepsilon(X^{m^{\ell}}_{T}-X_T)-
    (X^{m^{\ell-1}}_{T}-X_T) \cdot \varepsilon(X^{m^{\ell-1}}_{T}-X_T)
  \end{align*}
  with  $\varepsilon : \RR^d \to \RR^d$ satisfying $\lim_{\abs{x} \to 0}
  \varepsilon(x) = 0$.  From Property (\ref{eq:P}), we  easily get 
  $$
  \sqrt{\frac{ m^{\ell}}{(m-1)T}}\left((X^{m^{\ell}}_{T}-X_T) \cdot
    \varepsilon(X^{m^{\ell}}_{T}-X_T)- (X^{m^{\ell-1}}_{T}-X_T) \cdot
    \varepsilon(X^{m^{\ell-1}}_{T}-X_T) \right)
  \xrightarrow[\ell\to\infty]{\PP}0.
  $$
  So,  we conclude from Lemma \ref{lemma} and  Theorem \ref{th-acc} that
  \begin{equation*}
    {\sqrt{\frac{m^{\ell}}{(m-1)T}}} \left(\psi( X^{m^{\ell}}_{T})-\psi(
      X^{m^{\ell-1}}_{T})\right) \Longrightarrow^{stably} \nabla
    \psi(X_T).U_T,\; \mbox{ as } \;\ell\rightarrow\infty.
  \end{equation*}
  Let $\eta > \kappa >0$. From the assumptions on $\psi$ together with Property \ref{eq:P}, we get 
  \begin{equation*}
    \sup_{\l\geq 0}\EE\left[\left|\left(\frac{m^{\ell}}{(m-1)T}\right)^{\delta/2}
    \left(\psi( X^{m^{\ell}}_{T})-\psi(
      X^{m^{\ell-1}}_{T})\right)^\delta Y \right|^{1+\kappa}\right]<\infty,
  \end{equation*}
  which yields the uniform integrability of the family
  $\left(\left(\frac{m^\l}{(m-1)T}\right)^{\delta/2}
    \left(\psi(X^{m^{\ell}}_{T})-\psi(X^{m^{\ell-1}}_{T})\right)^\delta Y
  \right)_\l$. The conclusion easily follows.
\end{proof}

\bibliographystyle{abbrvnat}
\bibliography{biblio}
\end{document}